\documentclass[a4paper,10pt]{article}
\usepackage[utf8]{inputenc}
\usepackage{amsmath,amsxtra}
\usepackage[varg]{txfonts}
\usepackage[unicode]{hyperref}
\usepackage{framed}
\usepackage[amsmath,hyperref,framed,thmmarks]{ntheorem}
\usepackage{mathrsfs}
\usepackage{graphicx}
\usepackage{bbm}

\theoremseparator{.}
\makeatletter
\renewcommand{\newframedtheorem}[1]{%
\theoremprework{\framed\vspace{-1.5ex}}%
\theorempostwork{\vspace{-2ex}\endframed}%
\newtheorem@i{#1}%
}
\makeatother
\newframedtheorem{theorem}{Theorem}[section]
\newframedtheorem{proposition}[theorem]{Proposition}
\newframedtheorem{lemma}[theorem]{Lemma}
\newframedtheorem{corollary}[theorem]{Corollary}
\theorembodyfont{\upshape}
\theoremsymbol{$\Diamond$}
\newtheorem{remark}[theorem]{Remark}
\makeatletter
\newtheoremstyle{proof}%
{\item[\theorem@headerfont\hskip\labelsep ##1\theorem@separator]}%
{\item[\ifx\@empty##3\else\theorem@headerfont\hskip \labelsep ##1\ of\ ##3\theorem@separator\fi]}
\makeatother
\theoremheaderfont{\itshape}
\theoremstyle{proof}
\theoremseparator{.}
\theoremsymbol{\ensuremath{\Box}}
\newtheorem{proof}{Proof}
\theoremheaderfont{\sffamily\bfseries}\theorembodyfont{\sffamily\upshape}
\theoremstyle{plain}
\theoremnumbering{alph}
\theoremsymbol{$\Diamond$}
\theoremprework{\small}
\theorempostwork{\normalsize}

\hypersetup{
breaklinks=true,
colorlinks=true,
linkcolor=blue,
citecolor=blue,
urlcolor=blue,
}

\allowdisplaybreaks

\def\bfseries{\fontseries \bfdefault \selectfont \boldmath}

\makeatletter
\def\@fnsymbol#1{\ensuremath{\ifcase#1\or *\or **\or {**}* \or {**}{**}\else\@ctrerr\fi}}
\makeatother

\newcommand{\abs}[1]{\ensuremath{\left|#1\right|}}
\newcommand{\br}[1]{\ensuremath{\left(#1\right)}}
\newcommand{\bv}[1]{\ensuremath{\left[#1\right]}}
\renewcommand{\d}{\ensuremath{\,\mathrm{d}}}

\newcommand{\lnorm}[2][2]{\ensuremath{\left\|#2\right\|_{L^{{\scriptsize#1}}}}}
\newcommand{\norm}[1]{\ensuremath{\left\|#1\right\|}}
\newcommand{\seminorm}[1]{\ensuremath{\left[#1\right]}}

\newcommand{\seqn}[2][k]{\ensuremath{\br{{#2}_{#1}}_{#1\in\N}}}

\newcommand{\set}[1]{\ensuremath{\left\{#1\right\}}}
\newcommand{\sets}[2]{\ensuremath{\left\{\left.#1\,\right|#2\right\}}}
\newcommand{\sett}[2]{\ensuremath{\left\{#1\left|\,#2\right.\right\}}}
\renewcommand{\sp}[1]{\ensuremath{\left\langle#1\right\rangle}}

\numberwithin{equation}{section}

\DeclareFontFamily{U}{matha}{\hyphenchar\font45}
\DeclareFontShape{U}{matha}{m}{n}{
      <5> <6> <7> <8> <9> <10> gen * matha
      <10.95> matha10 <12> <14.4> <17.28> <20.74> <24.88> matha12
      }{}
\DeclareSymbolFont{matha}{U}{matha}{m}{n}
\DeclareFontFamily{U}{mathx}{\hyphenchar\font45}
\DeclareFontShape{U}{mathx}{m}{n}{
      <5> <6> <7> <8> <9> <10>
      <10.95> <12> <14.4> <17.28> <20.74> <24.88>
      mathx10
      }{}
\DeclareSymbolFont{mathx}{U}{mathx}{m}{n}
\DeclareMathSymbol{\oast}{2}{matha}{"66}
\DeclareMathSymbol{\bigoast}{1}{mathx}{"C6}

\renewcommand{\a}{\ensuremath{\alpha}}

\renewcommand{\C}{\ensuremath{\mathbb{C}}}

\newcommand{\Cia}[1][1]{\ensuremath{C_{\mathrm{ia}}^{#1}}}

\newcommand{\D}[1]{{#1}(u+w)-{#1}(u)}
\newcommand{\dd}[1][\tau]{{\textstyle\frac{\d}{\d#1}}}
\newcommand{\dg}{{\gamma}'}
\newcommand{\dgt}{\gt'}
\renewcommand{\dh}{{h}'}
\DeclareMathOperator{\diam}{diam}
\DeclareMathOperator{\dist}{dist}

\newcommand{\E}[1][p,q]{\ensuremath{\mathrm{TP}^{(#1)}}}
\newcommand{\eps}{\ensuremath{\varepsilon}}

\newcommand{\Fe}[1][]{\ensuremath{F_{#1}}}
\newcommand{\fracabs}[1]{\frac{#1}{\abs{#1}}}
\newcommand{\g}{\gamma}
\newcommand{\gp}{g^{(p)}}
\newcommand{\gt}[1][\tau]{\gamma_{#1}}
\renewcommand{\G}[1][p]{\ensuremath{\mathcal G^{(#1)}}}

\DeclareMathOperator{\image}{image}

\newcommand{\N}{\ensuremath{\mathbb{N}}}

\renewcommand{\P}[1][\dg(u)]{P_{#1}}
\newcommand{\PP}[1][\dg(u)]{P_{#1}^\perp}
\newcommand{\PPT}[1][\dg_\tau(u)]{P_{#1}^\perp}

\newcommand{\Q}[1][p]{Q^{(#1)}}
\newcommand{\Qe}[1][p]{Q^{(#1)}}
\newcommand{\R}{\ensuremath{\mathbb{R}}}
\renewcommand{\Re}[1][]{R_{#1}^{(p)}}

\renewcommand{\rho}{\ensuremath{\varrho}}
\newcommand{\rzd}{\ensuremath{(\R/\Z,\R^n)}}
\newcommand{\s}{\ensuremath{\sigma}}

\newcommand{\T}[1]{\triangle{#1}}

\renewcommand{\th}{\ensuremath{\vartheta}}

\newcommand{\W}[1][(p-1)/q-1,q]{\ensuremath{W^{\scriptstyle #1}}}
\newcommand{\Wia}[1][(p-1)/q,q]{\ensuremath{W_{\mathrm{ia}}^{#1}}}
\newcommand{\Wint}{\int_{\R/\Z}\int_{-1/2}^{1/2}}
\newcommand{\Wir}[1][(p-1)/q,q]{\ensuremath{W_{\mathrm{ir}}^{#1}}}
\newcommand{\Z}{\ensuremath{\mathbb{Z}}}

\title{Regularity theory for tangent-point energies: \\ The non-degenerate sub-critical case}
\author{%
 Simon Blatt\thanks{Mathematics Institute, Zeeman Building, University of Warwick, Coventry CV4 7AL, United Kingdom, \url{S.Blatt@warwick.ac.uk}}
 \and
 Philipp Reiter\thanks{Fakult\"at f\"ur Mathematik der Universität Duisburg-Essen, Forsthausweg 2, 47057 Duisburg, Germany, \url{philipp.reiter@uni-due.de}}}

\begin{document}
\maketitle
\begin{abstract}
In this article we introduce and investigate a new two-parameter family
of knot energies $\E$ that contains the tangent-point energies. These
energies are obtained by decoupling the exponents in the 
numerator and denominator of the integrand in the original definition of
the tangent-point energies.

We will first characterize the curves of finite
energy $\E$ in the sub-critical range $p\in(q+2,2q+1)$ and see that those
are all injective and regular curves in the Sobolev-Slobodecki{\u\i} space
$\W[(p-1)/q,q]\rzd$. We derive a formula for the first variation
that turns out to be a non-degenerate elliptic operator for the special 
case $q=2$ --- a fact that seems not to be the case for the original 
tangent-point energies. 
This observation allows us to prove that stationary points
of $\E[p,2]+\lambda\,\text{length}$, $p\in(4,5)$, $\lambda>0$,
are smooth --- so especially all local minimizers are smooth.
\end{abstract}
\tableofcontents

\section{Introduction}

%
%

Strzelecki and von der Mosel~\cite{SM7} introduced us to the crew of a space shuttle
travelling with constant speed through the universe on an unknown closed 
loop~$\Gamma$ of length $L$.
With the aid of their instruments they are able to measure at time~$t$ the
ratio of the squared distances $\abs{\Gamma(s)-\Gamma(t)}^2$ from any previous position $\Gamma(s)$, $s\in[0,t]$,
to the distance of the current tangent line $\ell(t)=\Gamma(t)+\R\Gamma'(t)$ from that previous position $\Gamma(s)$, i.~e.\@
\begin{equation*}
 2r_\Gamma(t,s) := \frac{\abs{\Gamma(s)-\Gamma(t)}^2}{\dist\br{\ell(t),\Gamma(s)}}.
\end{equation*}
Interestingly, the astronauts can  gain essential topological information 
and regularity properties from the integral mean of a suitable inverse power 
of all these data, more precisely from
\begin{equation}\label{eq:tp-StrzvdM}
 \mathscr E_q(\Gamma) := \iint_{[0,L]^2} \frac{\d s\d t}{r_\Gamma(t,s)^q}, \qquad q\ge2.
\end{equation}

During a hazardous maneuver
in the southern Andromeda Galaxy
the space craft unfortunately crashed, so the astronauts have to purchase a new one.
The manufacturer meanwhile changed the model
which now measures the ratio
\begin{equation}\label{eq:ratio}
 \tilde r_\Gamma^{(p,q)}(t,s) := \frac{\abs{\Gamma(s)-\Gamma(t)}^p}{\dist\br{\ell(t),\Gamma(s)}^q}
  \qquad\text{for predefinable variables } p,q\ge1
\end{equation}
and praises his innovation for giving more flexibility by choosing the ``power parameters'' $p$ and $q$.
He promises that the integral
\begin{equation}\label{eq:tp}
 \E(\Gamma) := \iint_{[0,L]^2} \frac{\d s\d t}{\tilde r_\Gamma^{(p,q)}(t,s)}
\end{equation}
yields far more information on the topology and regularity of the loop~$\Gamma$
and claims to have obtained particularly good results for $q=2$ and $p$ somewhere between~$4$ and~$5$.
Is he right?

%
%

We will see that for certain parameters the energy $\E$
is a knot energy. 
The notion of \emph{knot energies} goes back to Fukuhara~\cite{fete} and O'Hara~\cite{oha:en}.
The general idea is to search for a ``nicely shaped'' representative in a given knot class
having strands being widely apart and being preferably smooth. More precisely, a knot energy is a functional
that is (i)~bounded below and (ii)~\emph{self-repulsive} (or, synonymously, \emph{self-avoiding}),
i.~e.\@ it blows up on embedded curves converging to a curve with a self-intersection
(with respect to a suitable topology)~\cite[Def.~1.1]{oha:en-kn}.

Knot energies are the central object of the so-called \emph{geometric knot theory} which aims at
investigating geometric properties of a given knotted curve in order to gain information on its knot type.
They also form a subfield of \emph{geometric curvature energies} which include geometric integrals measuring
smoothness and bending for objects that \emph{a priori} do not have to be smooth.

Knot energies can help to model repulsive forces of fibres.
The original \emph{Gedankenexperiment} by Fukuhara~\cite{fete}
was the deformation of a thin fibre charged with electrons lying in
a viscous liquid. There is indication for DNA molecules seeking to attain a minimum state of a suitable energy~\cite{moffatt}.
Attraction phenomena may also be modeled by a corresponding positive gradient flow~\cite{alt-felix}.

The first knot energy on smooth curves goes back to O'Hara~\cite{oha:en} who in 1991 defined the
functional that was called \emph{M\"obius energy} later on by Freedman, He, and Wang~\cite{fhw}.
It corresponds to the element $E^{2,1}$ of the two-parameter family of functionals
\begin{equation}
  E^{\a,p}(\g):=
  \int_{\R/\Z}\int_{-1/2}^{1/2}
  \br{\frac{1}{\abs{\g(u+w)-\g(u)}^\a} - \frac{1}{D_\g(u+w,u)^\a}}^p
  \abs{\dg(u+w)}\abs{\dg(u)} \d w\d u \label{eq:a}
\end{equation}
which O'Hara~\cite{oha:fam-en,oha:en-func} introduced shortly after.
Here $\a,p>0$, and $\g\in C^{0,1}\rzd$.
The quantity $D_\g(u+w,u)$ measures the intrinsic distance between $\g(u+w)$ and $\g(u)$ on the curve~$\g$.
Of particular interest is the subfamily
\begin{equation}\label{eq:Ea}
 E^{(\a)}:=E^{\a,1} \qquad \text{for } \a\in[2,3).
\end{equation}
There are numerous contributions concerning topology~\cite{oha:fam-en,oha:en-func,fhw},
regularity \cite{acfgh,oha:en,fhw,he:elghf,blatt-reiter,reiter:rkepdc,reiter:rtme},
and the corresponding gradient flow~\cite{he:elghf,blatt:gfm,blatt:gfoh}.
Numerical experiments have been carried out in~\cite{kusner-sullivan}, error estimates have been obtained in~\cite{RSi,RW}.

Another famous example of a knot energy is the reciprocal of \emph{thickness}
which can be characterized by means of the \emph{global radius of curvature}~$\rho[\gamma]$
defined by Gonzalez and Maddocks~\cite{gonz-madd}.
This leads to the concept of \emph{ideal knots},
minimizers of the \emph{ropelength}
(the quotient of length and thickness)
within a prescribed isotopy class.
Existence is discussed in~\cite{gmsm,cakusu,glll}
while the question of regularity turns out to be rather involved~\cite{schur-vdm:elg,schur-vdm:ideal,cfks}.
In fact, an {explicit} analytical characterization of
the shape of a (non-trivial) ideal knot has not been found yet,
so the state of the art is discretization and numerical visualization,
cf.~\cite{ACPR,cprvis,carl-gerl,carlen-et-al,gerlach:diss,gms,smutny}.
\emph{Maximizing} length for prescribed thickness
on the two-dimensional sphere~$\mathbb S^{2}$
leads to an interesting packing problem, see
Gerlach and von der Mosel~\cite{GvdMa,GvdM}.

Substituting some of the minimizations in the definition of thickness as proposed in~\cite[Sect.~6]{gonz-madd},
one derives three families of integral-based energies,
namely
\begin{align*}
 \mathscr U_p(\g) &:= \br{\int_{\R/\Z}\frac{\d s}{\inf_{\R/\Z\setminus\set s}\rho[\g](s,\cdot)^p}}^{1/p}, \\
 \mathscr I_p(\g) &:= \iint_{(\R/\Z)^2} \frac{\d s\d t}{\rho[\g](s,t)}, \\
 \mathscr M_p(\g) &:= \iiint_{(\R/\Z)^3} \frac{\d s\d t\d\sigma}{R(s,t,\sigma)^p},
\end{align*}
where $R(s,t,\sigma)$ denotes the radius of the circle passing through the three points $\g(s)$, $\g(t)$, $\g(\sigma)$.
These functionals have been thoroughly investigated by Strzelecki and von der Mosel~\cite{SM3},
Strzelecki, Szuma\'nska and von der Mosel~\cite{SM4,SM5},
and Hermes~\cite{hermes}.
The energy spaces are discussed in~\cite{blatt:imcc}.
Energies for higher-dimensional objects are considered in
Strzelecki and von der Mosel~\cite{SM1,SM2,SM6},
Kolasi\'nski~\cite{kola,kola2}, and
Kolasi\'nski, Strzelecki, and von der Mosel~\cite{KSM}.

The \emph{tangent-point energies}~\eqref{eq:tp-StrzvdM} are a variant of these ``three-point circle'' based functionals.
One just uses the radius of the smallest circle tangent to 
one point and going through another point on the curve instead
of the radius of the smallest circle going through three points on the curve.
The resulting energies already appeared as~$U_{p,2}[\mathcal C]$ in the article by Gonzalez and Maddocks~\cite[Sect.~6]{gonz-madd}.
Sullivan~\cite{sullivan:appr-rope} used these functionals to approach \emph{ropelength}.
In contrast to these classical energies,
the integrand of the generalized energies introduced in this article~\eqref{eq:tp} 
bare such an appealing geometric interpretation. 
But we will see that they have nicer analytic properties, basically due to the
fact that their first variation leads to non-degenerate elliptic operator.

%
%

Before presenting the results of this article, let
us briefly review the main known results on the 
tangent-point energies defined in~\eqref{eq:tp-StrzvdM}.  
The most striking observation Strzelecki and von der Mosel made in their
seminal paper~\cite{SM7}, is that if $\mathscr E_q(\Gamma)$, $q\ge2$, is finite then 
the image of $\Gamma$ is a one-dimensional topological manifold~\cite[Thms.~1.1 and~1.4]{SM7}
of class $C^{1,1-2/q}$ if $q>2$~\cite[Thm.~1.3]{SM7}.
The proof of this and all the other main results in the paper is based on 
exploiting a decay estimate of Jones' beta numbers. 
Still for $q>2$, they gave an explicit upper bound on the Hausdorff distance 
of two given curve in terms of their tangent-point energies implying ambient isotopy~\cite[Thm.~1.2]{SM7}.
Moreover, they could prove that in this case $\mathscr E_q$ is a knot energy~\cite[Prop.~5.1]{SM7}, which improves an earlier result by Sullivan~\cite[Prop.~2.2]{sullivan:appr-rope} that requires higher regularity.
The energy even is a \emph{strong} knot energy, i.~e.\@
for given bounds on energy and length there are only finitely many knot types having a representative that satisfies these bounds.

In fact one can strengthen the above-mentioned result of Strzelecki and
von der Mosel and show that $\mathscr E_q(\Gamma)$ is finite if and only
if the image of $\Gamma$ is an embedded manifold of class
$W^{2-1/q,q}\subset C^{1,1-2/q}$, see~\cite[Cor.~1.2]{blatt:estp}.
Results for higher-dimensional analoga to~$\mathscr E_q$ can be found
in~\cite{SM8,blatt:estp,KSM}.

%
%

As $\E$ is (increasing in~$p$ and) decreasing in~$q$, the results by Strzelecki and von der Mosel~\cite{SM7} 
immediately carry over to
the $\E$-functionals~\eqref{eq:tp} with $p\ge2q$, $p\ge4$ via
\begin{equation}\label{eq:subfamily}
 \mathscr E_{p/2}=2^{p/2}\E[p,p/2],\qquad p\ge4.
\end{equation}
In fact we will show in Appendix~\ref{sect:beta}, that even for the full sub-critical
range
\begin{equation}\label{eq:sub-critical-range}
 p\in(q+2,2q+1), \qquad q>1
\end{equation}
the arguments in~\cite{SM7} can easily be adapted leading to
self-repulsiveness of the energies and H\"older regularity of the
first derivative. As in \cite{SM7} this can be used to
show for example that these energies are strong and that minimizers
exist in every knot class.

%
%

Since the arguments in \cite{SM7} are quite involved and technical,
we will present a completely independent and fast approach to these 
type of questions for curves that are \emph{a~priori} injective, 
continuously differentiable and parametrized by arc-length. 
This approach is based on techniques developed in 
\cite{blatt:estp}.

\begin{figure}[h]\label{fig:range}
 \includegraphics{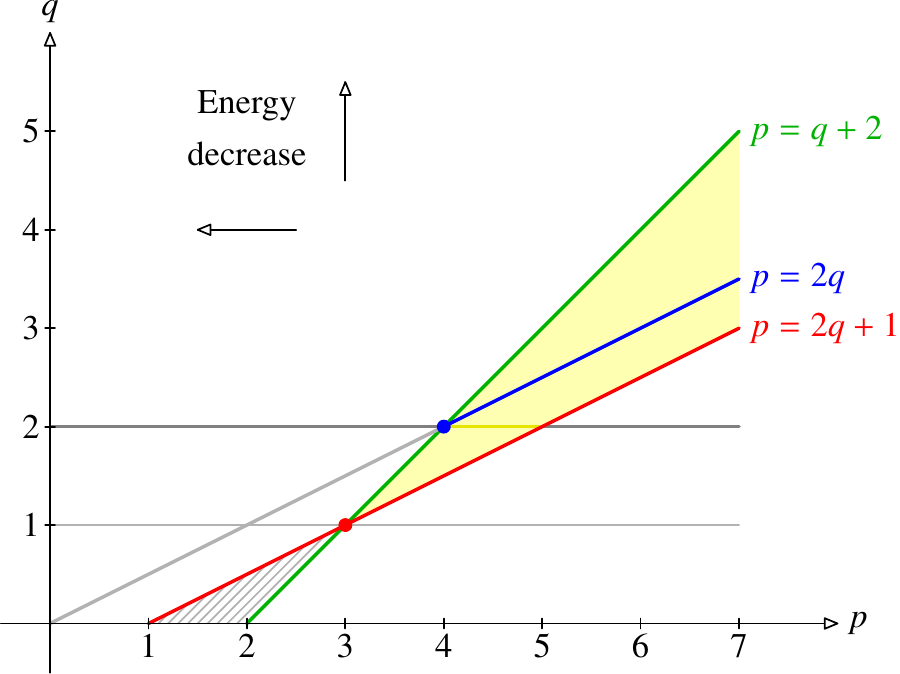}
 \centering
 \caption{The range of the tangent-point functionals~\eqref{eq:tp}.
 Above the green line, there is no self-repulsion (Remark~\ref{rem:self-repulsion}). On the red line and below, the functionals are singular, i.~e.\@ they identically take the value $+\infty$ (Proposition~\ref{prop:singular}).
 They are knot energies on the yellow area~\eqref{eq:sub-critical-range}
 (Proposition~\ref{prop:strong-knot-energy}).
 For this range of parameters, we can classify all curves of finite energy (Theorem~\ref{thm:energ-space}) 
 and prove the existence
 of minimizers within knot classes (Theorem~\ref{thm:existence}) and of the first variation (Theorem~\ref{thm:first-var}) on these spaces.
 For $\E[p,2]$, $p\in(4,5)$, marked by the yellow line~\eqref{eq:regularity-range},
 we obtain regularity of stationary points (Theorem~\ref{thm:smooth}).
 In the hatched area we find the strange behavior that the functionals take finite energy on polygons but not on closed $C^{3}$-curves (Remark~\ref{rem:strange}).
 The blue line visualizes the one-parameter subfamily~\eqref{eq:tp-StrzvdM},
 the original integral tangent-point energies.}
\end{figure}

The first result we get for the subcritical range of parameters~\eqref{eq:sub-critical-range}
is the following characterization of curves of finite energy
among all injective $C^1$-curves 
parametrized by arc-length.

\begin{theorem}[Energy spaces]\label{thm:energ-space}
 Assume~\eqref{eq:sub-critical-range} and
 let $\g \in C^1(\R/\Z,\R^{n})$ be an injective curve parametrized by 
 arc-length.
 Then $\E(\g)<\infty$ if and only if~$\g \in W^{(p-1)/q,q}$. Moreover,
 one then has, for constants $C,\beta>0$ depending on $p,q$ only,
 \begin{equation}\label{eq:energy-bound}
  \|\g\|^q_{\W[(p-1)/q,q]} \leq C\br{\E(\g) + \E(\g)^\beta}.
 \end{equation}
\end{theorem}

\begin{remark}[Initial regularity]\label{rem:C1}
Note that our method of proof works entirely
without using the techniques by von der Mosel and Strzelecki~\cite{SM7} --- if one always assumes curves to be continuously differentiable
as stated in the preceding Theorem~\ref{thm:energ-space}.

However, the requirement of initial $C^{1}$-regularity can be omitted
as the image of finite-energy curves is an embedded $C^{1,\alpha}$-manifold by Theorem~\ref{thm:topo}, which is easily derived from~\cite{SM7}.
This shows that the energy of an arbitrary absolutely continuous curve is finite if and only if
its image is an embedded manifold of class $W^{(p-1)/q,q}$.
\end{remark}

%
%

We will then show how to combine Theorem~\ref{thm:energ-space}
with a bi-Lipschitz estimate to obtain the existence of minimizers in every 
knot class:

\begin{theorem}[Existence of minimizers within knot classes]\label{thm:existence}
 Assume~\eqref{eq:sub-critical-range}. \\ Then, in any knot class there is a minimizer of 
 $\E$ for~\eqref{eq:sub-critical-range} among all injective, regular curves 
 $\g \in C^1(\R / \Z, \R^n)$.
\end{theorem}

%
%

In order to study stationary points of the
energy, we derive a formula for the first variation
on the space of injective and regular curves of finite energy.
To shorten notation we abbreviate
\begin{equation}\label{eq:short-notation}
 \T\bullet:=\D\bullet.
\end{equation}
Let 
\begin{align} 
\P a := \sp{a,\fracabs{\dg(u)}}\fracabs{\dg(u)}, \qquad
\PP a := a - \P a \qquad\text{for }a\in\R^{n} 
\end{align}
be the projection onto the tangential and normal part along $\g$
respectively.

\begin{theorem}[First variation]\label{thm:first-var}
 For $p,q$ satisfying~\eqref{eq:sub-critical-range}
 let $\g\in\W[(p-1)/q,q]\rzd$ be injective and parametrized by arc-length.
 Then, for any $h\in\W[(p-1)/q,q]\rzd$, the first variation of~$\E$ at~$\g$ in direction~$h$
 exists and amounts to
 \begin{align}
  &\delta\E(\g,h) \nonumber\\
 &=\Wint \Bigg\{q\frac{\sp{\PP\br{\T\g-w\dg(u)},\T h-\sp{\T\g,\dg(u)}\dh(u)}}{\abs{\T\g}^p} \abs{\PP\br{\T\g-w\dg(u)}}^{q-2} \nonumber\\
 &\qquad\qquad\qquad{} - p\frac{\abs{\PP\br{\T\g}}^q\sp{\T\g,\T h}}{\abs{\T\g}^{p+2}} \label{eq:dE}\\
 &\qquad\qquad\qquad{} + \frac{\abs{\PP\br{\T\g}}^q}{\abs{\T\g}^{p}}\br{\sp{\dg(u),\dh(u)}+\sp{\dg(u+w),\dh(u+w)}} \Bigg\}.\nonumber
 \end{align}
\end{theorem}

Be aware that, in contrast to O'Hara's knot energies,
we do not need a principal value to express the first variation here.
The same situation applies to the
integral Menger curvature functionals, see Hermes~\cite{hermes}.

In this article, we only calculate the first variation  at
arc-length parametrized curves
in order to make the proof as simple as possible.
However, adapting the techniques from~\cite{blatt-reiter2}, one can even derive
continuous differentiability of $\E$ on the set of all injective regular
curves in $\W[(p-1)/q,q]\rzd$.

For the non-degenerate case $q=2$ we then finally 
study the regularity of stationary points of finite energy,
i.~e.\@ curves $\g \in \W[(p-1)/2,2]\rzd$ where
\begin{equation}\label{eq:regularity-range}
 p\in(4,5), \qquad q=2.
\end{equation}
We will see that
those are smooth --- which in a sense is a justification for
inventing these new knot energies in the first place.

%
%

\begin{theorem}[Stationary points of ${\E[p,2]}$ are smooth]\label{thm:smooth}\hfill\\
 For $p\in(4,5)$, let $\g\in\W[(p-1)/2,2]\rzd$ be a stationary point of $\E[p,2]$ with respect to fixed length, injective and parametrized by arc-length.
 Then $\g\in C^{\infty}$.
\end{theorem}

Surprisingly, the proof of this theorem up to some new technical difficulties 
roughly follows the lines of the proof of the analogous result for 
O'Hara's energies $E^{\alpha,1}$ for $\alpha \in (2,3)$ ---
yet another indication that the two families of energies $\E$
and $E^{\alpha,p}$ are not too different from the perspective of an analyst.

In Proposition~\ref{prop:Q} we will see that, for $q=2$,
the highest term in the Euler-Lagrange equation is an elliptic operator
of order $p-1$.
We will show that the remainder 
consist of terms having a common form (Lemma~\ref{lem:struc-rem})
and is of lower order. This allows is to apply a bootstrapping argument
to show that critical points are smooth and thus prove Theorem~\ref{thm:smooth}.

Let us stress once more that we do \emph{not} expect the latter 
result to carry over to other parameters in~\eqref{eq:sub-critical-range}.
This is due to the fact that
the first variation should then be a degenerate elliptic operator.

\begin{remark}[The critical case $p=q+2$]\label{rem:critical}
 Although we generally restrict to~\eqref{eq:sub-critical-range},
 our results partially also apply to the \emph{critical case}
 $p=q+2$.

 This holds true for the characterization of energy spaces
 in Theorem~\ref{thm:energ-space} except for Estimate~\eqref{eq:energy-bound}
 and the derivation of the first variation in Theorem~\ref{thm:first-var}
 where we additionally have to claim $\g,h\in C^{1}$.

 However, the proofs of both Theorem~\ref{thm:existence} and
 Theorem~\ref{thm:smooth} fundamentally rely on $p>q+2$.
 In the light of corresponding results for the M\"obius energy
 $E^{2,1}$~\cite{fhw,brs} we expect these situation to be much more involved.
\end{remark}

To make the article as accessible as possible,
we present the two main tools used in the bootstrapping argument,
namely chain and product rules for fractional
Sobolev spaces, in Appendix~\ref{sect:fractional}.
Furthermore, a sketch on how to prove that finiteness
of the energy implies embeddedness (Theorem~\ref{thm:topo}) 
can be found in Appendix~\ref{sect:beta}.

%
%

Let us bring the energies into the form we will work with from now on.
Observing that
\begin{align*}
 &\dist\br{\ell(u),\gamma(u+w)} = \abs{\PP\br{\D\g}} \\
 &{= \sqrt{\abs{\D\g}^2-\abs{\sp{\D\g,\dg(u)}}^2}}
\end{align*}
and taking into account absolutely continuous curves (of arbitrary regular parametrization),
the functional~\eqref{eq:tp} may be rewritten as
\begin{equation}\label{eq:tp-reg2}
 \E(\g) = \int_{\R/\Z}\int_{-1/2}^{1/2} \frac{\abs{\PP\br{\D\g}}^q}{\abs{\D\g}^p} \abs{\dg(u+w)}\abs{\dg(u)}\d w\d u.
\end{equation}
It will be crucial for the estimates later on, that
\begin{equation}\label{eq:proj}
 \PP\br{\D\g} = \PP\br{\D\g-w\dg(u)},
\end{equation}
so, for $\g\in C^{1,1}$, the integrand in~\eqref{eq:tp-reg2} 
behaves like $\mathcal O\br{\abs w^{2q-p}}$ as $w\to0$.

We will use \emph{Sobolev-Slobodecki{\u\i} spaces} in the form they already 
appeared in~\cite{blatt:bre}.
For the readers' convenience we briefly recall their definition and
some basic properties.
Let $f\in L^2\rzd$. For $s\in(0,1)$ and $\rho\in[1,\infty)$ we define the seminorm
\begin{equation}\label{eq:Wsemi}
 \seminorm{f}_{\W[s,\rho]} := \br{\int_{\R/\Z}\int_{-1/2}^{1/2} \frac{\abs{f(u+w)-f(u)}^\rho}{\abs w^{1+\rho s}} \d w\d u}^{1/q}.
\end{equation}
Now let $W^{k,\rho}\rzd$, $k\in\N\cup\set0$, denote the usual Sobolev space
(recall $W^{0,\rho}:=L^\rho$) and
\[ W^{k+s,\rho}\rzd := \sett{f\in W^{k,\rho}\rzd}{\norm f_{\W[k+s,\rho]}<\infty} \]
be equipped with the norm
\[ \norm f_{\W[k+s,\rho]} := \norm f_{\W[k,\rho]} + \seminorm{f^{(k)}}_{\W[s,\rho]}. \]
Without further notice we will frequently use the embedding
\begin{equation}\label{eq:Wembed}
 \W[k+s,\rho]\rzd\hookrightarrow C^{k,s-1/\rho}\rzd, \qquad s\in(\rho^{-1},1).
\end{equation}
We will denote by $\Cia[]$ resp.\@ $\Wia[]$
\underline injective (embedded) curves parametrized by \underline arc-length
and by $\Wir[]$
\underline injective \underline regular curves.
As usual, a curve is said to be \emph{regular} if there is some $c>0$ such that $\abs\dg\ge c$ a.~e.
Constants may change from line to line.

\paragraph{Acknowledgements.}
The first author was supported by Swiss National Science Foundation Grant Nr.~200020\_125127 and the Leverhulm trust.
The second author was supported by DFG Transregional Collaborative Research Centre SFB~TR~71.
This project was initiated during the ESF Research Conference `Knots and Links: From Form to Function',
2 -- 8 July 2011, at the Mathematical Research Center `Ennio De Giorgi', Pisa, Italy.
\section{Energy space}\label{sect:es}

The main aim of this section is to characterize in some sense the 
domain of the energies $\E$ in the range
\begin{equation}\tag{\ref{eq:sub-critical-range}}
 p\in(q+2,2q+1), \qquad q>1
\end{equation}

and prove the existence of minimizers using this result.

We will see that these are the 
only parameters for which the energies are both self-repulsive
and well-defined
in the sense that there exist closed curves of finite energy,
but not scaling invariant.

\begin{remark}[Not a knot energy if $p<q+2$]\label{rem:self-repulsion}
 Let us give an example that shows that we do not get
 a bi-Lipschitz estimate for injective curves if $p<q+2$.
 Consider the curves $u\mapsto (u,0,0)$ and $u\mapsto (0,u,\delta)$ for $u\in[-1,1]$, $\delta>0$.
 The interaction of these strands leads to the $\E$-value
 \begin{equation}\label{eq:cross}
 \begin{split}
  2\iint_{[-1,1]^2} \frac{\br{v^{2}+\delta^{2}}^{q/2}}{\br{u^2+v^2+\delta^2}^{p/2}}\d u\d v
  &\le 2\int_0^{\sqrt 2}\int_0^{2\pi} \frac{\br{r^{2}\sin^{2}\varphi+\delta^{2}}^{q/2}}{\br{r^2+\delta^2}^{p/2}}r\d\varphi\d r \\
  &\le 4\pi \int_0^{\sqrt 2}\br{r^2+\delta^2}^{\tfrac{q-p}2}r\d r.
 \end{split}
 \end{equation}
 The integral on the right-hand side is bounded for $\delta\searrow0$ if $p<q+2$.
 Using Proposition~\ref{prop:energ-space} below
 and the monotonicity of $\E[\cdot,q]$ for fixed $q$, it is easy to join the endpoints of the two strands via suitable arcs producing a
 family of ``figure eight''-like embedded smooth curves that does not lead to an energy blow-up as $\delta\searrow0$.
 Clearly this does not meet the requirements for a knot energy as mentioned in the introduction.
\end{remark}

The biggest difference here to the approach taken in
\cite{SM7} is that we will only look at curves parametrized
by arc-length which are \emph{a priori} $C^1$ and injective. 
It is surprising, that we will still
be able to prove by rather simple means that the subset of
these curves of bounded length and energy is compact 
in $C^1$ up to translations. This will follow from
our classification of curves of finite energy and
a bi-Lipschitz estimate which we will again prove using this classification.

We will use this together with the lower semi-continuity 
of the energies $\E$ with respect to convergence in $C^1$,
to show that these are strong knot energies that can be minimized
within each knot class --- without using one of the basic tools
in \cite{SM7}, the decay of Jones' beta numbers. 
But as in \cite{SM7} scaling is what makes our arguments work. More precisely:
that things get punished more by the energy, if they happen on a small scale.

\begin{remark}[Problem with non-injective curves]
To see that considering just injective 
curves might be an idea, let us repeat an observation that was already made in \cite{SM7}
for the classical tangent point energies.
As the value $\E(\g)$ only depends on the image of $\g$ and multiplicities, it is easy to construct a non-injective curve 
parametized by arc-length of finite energy that is moreover not $C^1$:
Take e.~g.\@ an open $C^2$-curve defined on $[0,\tfrac12]$. By Proposition~\ref{prop:energ-space} it has finite energy.
Traversing it once, changing the direction at the end-point, and then traversing it in the opposite direction,
produces a non-injective continuous parametrization on $\R/\Z$ of a one-dimensional manifold with boundary whose energy amounts to four times the original energy.
By the same reasoning, passing a curve $k$-times results in an energy increase by the factor $k^2$.
\end{remark}

To give a sufficient condition for an injective curve
in $C^1$ parametrized by arc length ---  which will also turn out 
to be necessary ---  we will use the following easy result
from~\cite{blatt:bre}:

\begin{proposition}[Bi-Lipschitz continuity {\cite[Lem.~2.1]{blatt:bre}}]\label{prop:bilip}\hfill\\
 Let $\g\in\Wia(\R/\Z,\R^n)$, $p\ge q+2$.
 Then there is a constant $C_\g$
 such that
 \begin{equation}\label{eq:bilip}
  \frac{\abs w}{C_\g}\le\abs{\D\g}\le\abs w \qquad\text{for all } u\in\R/\Z,w\in[-\tfrac12,\tfrac12].
 \end{equation}
\end{proposition}

In Proposition~\ref{prop:bi-Lipschitz} below we will provide a uniform
bi-Lipschitz estimate for curves of bounded $\E$-energy.

Now we are in the position to prove that curves in $\Wia(\R/\Z,\R^n)$
have finite energy:

\begin{proposition}[Sufficient regularity condition for $p\in[q+2,2q+1)$]\label{prop:energ-space}\hfill\\
 If $\g\in\Wia(\R/\Z,\R^n)$ is para\-metrized by arc-length, $q\geq 1$ and
 $p\in[q+2,2q+1)$ then $\E(\g)<\infty$.
\end{proposition}

\begin{proof}
 By~\eqref{eq:bilip} we derive as in~\cite{blatt:estp}
 \begin{align*}
  \E(\g) &= \int_{\R/\Z}\int_{-1/2}^{1/2} \frac{\abs{\PP\br{\D\g}}^q}{\abs{\D\g}^p}\d w\d u \\
  &\le {C_\gamma^p}\int_{\R/\Z}\int_{-1/2}^{1/2} \frac{\abs{\PP\int_0^1\dg(u+\th w)\d \th}^q}{\abs{w}^{p-q}}\d w\d u \\
  &= {C_\gamma^p}\int_{\R/\Z}\int_{-1/2}^{1/2} \frac{\abs{\int_0^1\PP\br{\dg(u+\th w)-\dg(u)}\d \th}^q}{\abs{w}^{p-q}}\d w\d u \\
  &\le {C_\gamma^p}\int_0^1\int_{\R/\Z}\int_{-1/2}^{1/2} \frac{\abs{{\dg(u+\th w)-\dg(u)}}^q}{\abs{w}^{p-q}}\d w\d u\d\th \\
  &\le {C_\gamma^p}\int_0^1\int_{\R/\Z}\int_{-\th/2}^{\th/2} \th^{p-q-1}\frac{\abs{{\dg(u+\tilde w)-\dg(u)}}^q}{\abs{\tilde w}^{p-q}}\d\tilde w\d u\d\th \\
  &\le {C_\gamma^p}\int_{\R/\Z}\int_{-1/2}^{1/2} \frac{\abs{{\dg(u+\tilde w)-\dg(u)}}^q}{\abs{\tilde w}^{p-q}}\d\tilde w\d u \\
  &= {C_\gamma^p}\seminorm\dg_{\W[(p-1)/q-1,q]}^q.
 \end{align*}
\end{proof}

To get a classification of all finite-energy curves in $\Cia$
we need to show that the inverse implication is true as well:

\begin{proposition}[Necessary regularity for finite energy]\label{prop:conttang}
 Let $\g\in C^{1}(\R/\Z,\R^n)$ be injective and 
 parametrized by arc-length with $\E(\g)<\infty$
 for~\eqref{eq:sub-critical-range}.
 Then $\g\in\W[(p-1)/q,q]$ and
 \begin{equation}\label{eq:energ-bound}
  \seminorm\dg_{\W[(p-1)/q-1,q]}^q 
  \le C\br{\E(\g)+ \E(\g)^{\beta}}
 \end{equation}
 where $C$ and $\beta>0$ depend only on $p$, $q$.
 Moreover,
 \begin{equation}\label{eq:hoelder-bound}
  \norm\dg_{C^{(p-2)/q-1}}^q 
  \le C\br{\E(\g)+ \E(\g)^{\beta}}.
 \end{equation}
\end{proposition}

\begin{proof}
 The estimate \eqref{eq:hoelder-bound} immediately follows from
 \eqref{eq:energ-bound} and Morrey's embedding theorem for
 fractional Sobolev spaces.
 
 The proof of \eqref{eq:energ-bound} uses the techniques from~\cite{blatt:estp}. By continuity
 we may choose some $\delta>0$ such that
 \begin{equation}\label{eq:sqrt2}
  \abs{\D{\dg}}\le\tfrac12\sqrt2 \quad\text{for all } u\in\R/\Z,w\in[-\delta,\delta].
 \end{equation}
 In fact we choose the biggest such constant, i.~e.\@ we assume that
 there are $u_{0}\in\R/\Z,w_{0}\in[-\delta,\delta]$ such that
 \begin{equation} \label{eq:sqrt22}
  \abs{\dg(u_{0}+w_{0})-\dg(u_{0})}=\tfrac12\sqrt2.
 \end{equation}

 This leads to
 \begin{align*}
  & \abs{\PP[\dg(u+w)]\br{\D\g}-\PP\br{\D\g}}^2 \\
  & = \abs{\sp{\D\g,\dg(u+w)}\dg(u+w) - \sp{\D\g,\dg(u)}\dg(u)}^{2} \\
  & = \abs{\sp{\D\g,\dg(u+w)}}^2 + \abs{\sp{\D\g,\dg(u)}}^2 \\
  &   \quad{} - 2\sp{\D\g,\dg(u)}\sp{\D\g,\dg(u+w)}\sp{\dg(u),\dg(u+w)} \\
  & = \abs{\sp{\D\g,\dg(u+w)}-\sp{\D\g,\dg(u)}}^2 \\
  &   \quad{} + \sp{\D\g,\dg(u)}\sp{\D\g,\dg(u+w)}\abs{\dg(u)-\dg(u+w)}^2 \\
  & \ge \abs{\dg(u)-\dg(u+w)}^2 \, w^2\int_0^1
         \underbrace{\sp{\dg(u+\th_1w),\dg(u)}}_{=1-\frac12\abs{\dg(u+\th_1w)-\dg(u)}^2\ge\frac34}\d\th_1
        \int_0^1\underbrace{\sp{\dg(u+\th_2w),\dg(u+w)}}_{\ge\frac34}\d\th_2 \\
  & \ge \tfrac9{16}w^2\abs{\D\dg}^2.
 \end{align*}
 This allows to estimate using $\abs{\D\g}\le\abs w$
 \begin{align*}
  \E(\g) &= \int_{\R/\Z}\int_{-1/2}^{1/2} \frac{\abs{\PP\br{\D\g}}^q}{\abs{\D\g}^p} \d w\d u\\
  &= \tfrac12\int_{\R/\Z}\int_{-1/2}^{1/2} \frac{\abs{\PP[\dg(u+w)]\br{\D\g}}^q + \abs{\PP\br{\D\g}}^q}{\abs{\D\g}^p} \d w\d u\\
  &\ge c_{p,q} \int_{\R/\Z}\int_{-\delta}^{\delta} \frac{\abs w^q\abs{\D\dg}^q}{\abs{\D\g}^{p}} \d w\d u\\
  &\ge c_{p,q} \int_{\R/\Z}\int_{-\delta}^{\delta} \frac{\abs{\D\dg}^q}{\abs w^{p-q}}\br{\frac{\abs w}{\abs{\D\g}}}^{p} \d w\d u\\
  &\ge c_{p,q} \int_{\R/\Z}\int_{-\delta}^{\delta} \frac{\abs{\D\dg}^q}{\abs{w}^{p-q}}\d w\d u
 \end{align*}
 which gives
 \begin{align*}
  \seminorm\dg_{\W[(p-1)/q-1,q]}^q
  &\le C_{p,q}\br{\E(\g)+\norm\dg_{L^\infty}\int_\delta^{1/2}\frac{\d w}{w^{p-q}}} \\
  &\le C_{p,q}\br{\E(\g)+\delta^{1+q-p}
  }.
 \end{align*}
 Unfortunately, this last estimate gets worse as $\delta$ gets small.
 We will derive a Morrey estimate for fractional 
 Sobolev space to estimate $\delta$ from below. More precisely, we will show  \begin{equation}\label{eq:Campanato}
  \lnorm[\infty]{\gamma'(\cdot+w) - \gamma'(\cdot)} \leq C \E(\g)^{1/q} |w|^\alpha
  \qquad\text{for all }w\in[-\tfrac12,\tfrac12]
 \end{equation}
 where $\alpha = (p-2)/q-1>0$.
 From~\eqref{eq:Campanato} we infer
 \begin{equation*}
  \tfrac12\sqrt2 \le C \E(\g)^{1/q} \delta^\alpha
 \end{equation*}
 which concludes the proof.

 To complete the argument, we sketch the proof of the Morrey estimate stated
 above. Let $\g'_{B_r (x)}$ denote the integral mean 
 of $\g'$ over $B_r(x)$. We calculate for $x \in \mathbb R / \mathbb Z$
 and $r \in (0,\delta)$ 
 \begin{align*}
  \frac 1 {2r}\int_{B_r(x)} |\gamma'(v)- \g'_{B_r(x)}| \d v
  &\leq \frac 1 {4r^2} \int_{B_r(x)} \int_{B_r(x)}| \gamma'(v) - \gamma'(u)| \d u \d v
  \\
  &\leq  \left(\frac 1 {4r^2} \int_{B_r(x)} \int_{B_r(x)}| \gamma'(v) - \gamma'(u)|^q \d u \d v\right)^{1/q}
  \\
  &\leq C r^{\alpha} \left(\int_{B_r(x)} \int_{B_r(x)} \frac{| \gamma'(v) - \gamma'(u)|^q} {|u-v|^{p-q}}\d u \d v\right)^{1/q} 
  \\
  &\leq C r^{\alpha} \E(\g)^{1/q}.
 \end{align*}
 The estimate~\eqref{eq:Campanato} now follows from this
 by standard arguments due to Campanato~\cite{campanato}.
 We choose two Lebesgue points $u,v \in \mathbb R / \Z$ of $\gamma'$
 with $r:=|u-v| \in (0,\tfrac\delta2)$. Then
 \begin{equation*}
  |\dg(u) - \dg(v)| 
  \leq \sum_{k=0}^\infty \left| \dg_{B_{r2^{1-k}}(u)} - \dg_{B_{r2^{-k}}(u)} \right|
  + \left| \dg_{B_{2r}(u)} - \dg_{B_{2r}(v)} \right|
  + \sum_{k=0}^\infty \left| \dg_{B_{r2^{1-k}}(v)} - \dg_{B_{r2^{-k}}(v)} \right|.
  \end{equation*}
  
  Since
  \begin{align*}
   \left| \dg_{B_{2r}(u)} - \dg_{B_{2r}(v)} \right| 
   &\leq \frac {\int_{B_{2r}(u)} |\dg(x) - \dg_{B_{2r}(u)}| \d x +
    \int_{B_{2r}(v)} |\dg(x) - \dg_{B_{2r}(v)}| \d x} {|B_{2r} (u) \cap B_{2r}(v)|}
    \\
    &\leq C |u-v|^\alpha \E(\g)^{1/q} 
  \end{align*}
  as $r=|u-v|$ and for all $y \in \mathbb R / \mathbb Z$, $R\in(0,\tfrac\delta2)$
  \begin{align*}
   \left| \dg_{B_{2R}(y)} - \dg_{B_{R}(y)} \right| 
   &\leq \frac {\int_{B_{R}(y)} |\dg(x) - \dg_{B_{2R}(y)}| \d x +
    \int_{B_{R}(y)} |\dg(x) - \dg_{B_{R}(y)}| \d x} {R}
    \\
    &\leq C R^\alpha \E(\g)^{1/q},
  \end{align*}
  we deduce that
  \begin{align*}
   |\dg(u) - \dg(v)| 
   \leq C \left( \sum_{0} ^\infty 2^{-k\alpha} + 1 +  \sum_{0} ^\infty 2^{-k\alpha}   \right)
   |u-v|^\alpha \E(\g)^{1/q}.
  \end{align*}
  Thus
  \begin{equation*}
   |\dg(u) - \dg(v)| \leq C |u-v|^\alpha \E(\g)^{1/q}
  \end{equation*}
 for all Lebesgue points of $\gamma'$ with $|u-v| \leq \frac\delta2$.

 Since Lebesgue points are dense and using the triangle inequality this 
 proves \eqref{eq:Campanato}.
\end{proof}

We assume $p<2q+1$ in the last proposition mainly
because~\eqref{eq:Wsemi} is not defined for $s\ge1$.
For general $p\ge 2q+1$ we nevertheless still have
\begin{equation}\label{eq:energ-bound*}\tag{\ref{eq:energ-bound}*}
 \int_{\R/\Z}\int_{-1/2}^{1/2} \frac{\abs{\D\dg}^q}{\abs{w}^{p-q}}\d w\d u
 \le C\br{\E(\g)+\E(\g)^{\beta}}.
\end{equation}
This enables us to derive the following result on what we 
want to call the \emph{singular range}: For these parameters the integrand
is so singular if it does not vanish completely, that the integral is
either equal to $0$ or infinite:

\begin{proposition}[Singular range $p\ge2q+1$]\label{prop:singular}
 For $p\ge 2q+1$, $q>1$, and an absolutely continuous $\g:\R/\Z\to\R$ we have
 $\E(\g)<\infty$ if and only if
 the image of $\g$ lies on a straight line.
\end{proposition}

\begin{proof}
 Applying~\eqref{eq:energ-bound*} to Brezis~\cite[Prop.~2]{brezis}
 reveals that $\dg$ is constant.
 Hence, $\g$ lies on a straight line.
\end{proof}

Proposition~\ref{prop:conttang} and Proposition~\ref{prop:energ-space}
prove Theorem~\ref{thm:energ-space}.

Using Proposition~\ref{prop:conttang} together with
the Arzel\`a-Ascoli theorem, we see that sets of 
curves in $\Cia(\mathbb R / \mathbb Z, \mathbb R^n)$
with a uniform bound on the energy are sequentially compact in $C^1$.
The next proposition will help us to show that the limits we get
are injective curves:

\begin{proposition}[Uniform bi-Lipschitz estimate]\label{prop:bi-Lipschitz}
 For every $M < \infty$ and~\eqref{eq:sub-critical-range} there is a constant $C(M,p,q)>0$ such 
 that the following is true: Every curve $\gamma\in \Cia(\mathbb R / \mathbb Z, \mathbb R^n)$
 parametrized by arc-length with
 \begin{equation}\label{eq:TP<M}
  \E(\gamma) \leq M
 \end{equation}
 satisfies the bi-Lipschitz estimate
 \begin{equation*}
  \abs{u-v} \leq C(M,p,q) \abs{\gamma(u) - \gamma(v)}
  \qquad \text{for all } u,v \in \mathbb R / \mathbb Z.
 \end{equation*}
\end{proposition}

We will give an easy proof that essential boils down to combining 
the regularity we get form Proposition~\ref{prop:conttang} with
a subtle scaling argument. The following lemma will be one of the
essential parts in the proof. To be able to state it,
we set for two arc-length parametrized curves $\gamma_i : I_i \rightarrow \mathbb R$, $i=1,2$,
$I_1,I_2$ open intervals,
\begin{align*}
 \E(\gamma_1, \gamma_2) &:= \E(\gamma_1) + \E(\gamma_2) + {} \\
 &{}\qquad +
 \int_{I_1} \int_{I_2} \br{\frac{\dist(\ell_1(t), \gamma_2(s))^q }{|\gamma_1(s)-\gamma_2(t)|^p} 
 + \frac{\dist(\ell_2(t),\gamma_1(s))^q }{|\gamma_2(s)-\gamma_1(t)|^p}}\d s\d t,
\end{align*}
where $\ell_i(\tau)=\g_{i}(\tau)+\R\dg_{i}(\tau)$ denotes line tangential to $\gamma_i$ at $\gamma_i(\tau)$, $i=1,2$.

We then have

\begin{lemma}\label{lem:last}
 Let $\alpha \in (0,1)$. For $\mu>0$ we let
 $M_\mu$ denote the set of all pairs $(\gamma_1, \gamma_2)$
 of curves $\gamma_i \in \Cia([-1,1],\R^n)$ 
 satisfying
 \begin{enumerate}
  \item $|\gamma_1(0) - \gamma_2 (0)| = 1$,
  \item $\g_{1}'(0)\perp\br{\gamma_1(0)-\gamma_2(0)} \perp \gamma_2'(0)$,
  \item $\|\gamma_{i}'\|_{C^{0,\alpha}} \leq \mu$, \qquad$i=1,2$.
 \end{enumerate}
 Then there is a $c=c(\a,\mu)>0$ such that
 \begin{equation*}
  \E(\gamma_1, \gamma_2) \geq c
  \qquad\text{for all }(\gamma_1, \gamma_2) \in M_\mu.
 \end{equation*}
\end{lemma}

\begin{proof}
 It is easy to see that $\E(\gamma_1, \gamma_2)$
 is zero if and only if both $\gamma_1$ and $\gamma_2$ 
 are part of one single straight line.
 We will show that $\E(\cdot,\cdot)$ attains its minimum
 on $M_\mu$. As $M_\mu$ does not contain straight lines by~(i), (ii),
 this minimum is strictly positive which thus proves the lemma.
 
 Let $(\gamma_1^{(n)},\gamma_2^{(n)})$ be a minimizing sequence in 
 $M_\mu$, i.~e.\@ we have
 \begin{equation*}
  \lim_{n\rightarrow \infty}\E (\gamma_1^{(n)},\gamma_2^{(n)})
  = \inf_{M_\mu} \E(\cdot,\cdot).
 \end{equation*}
 Subtracting $\g_{1}(0)$ from \emph{both} curves, i.~e.\@ setting
 \begin{equation*}
  \tilde \gamma_i^{(n)}(\tau) := \gamma_i^{(n)}(\tau) - \gamma_{1}(0),
  \qquad i = 1,2,
 \end{equation*}
 and using Arzel\`a-Ascoli we can pass to a subsequence such that
 \begin{equation*}
  \tilde \gamma_i^{(n)} \to \tilde \gamma_i\qquad\text{in }C^{1}.
 \end{equation*}
 Furthermore, $(\tilde \gamma_1, \tilde \gamma_2) \in M_\mu$
 since $M_\mu$ is closed under convergence in $C^1$.
 Since, by Fatou's lemma, the functional $\E$ is lower semi-continuous
 with respect to $C^1$ convergence, we obtain
 \begin{equation*}
  \E(\tilde \gamma_1, \tilde \gamma_2) 
  \leq \lim_{n\rightarrow \infty} \E (\tilde \gamma_1^{(n)}, \tilde \gamma_2^{(n)})
  = \lim_{n\rightarrow \infty} \E ( \gamma_1^{(n)}, \gamma_2^{(n)})
  =  \inf_{M_\mu} \E(\cdot,\cdot).
  \end{equation*}
\end{proof}

Let us use this lemma to give the

\begin{proof}[Propsition~\ref{prop:bi-Lipschitz}]
Applying Proposition~\ref{prop:conttang} to~\eqref{eq:TP<M} we obtain
\begin{equation*}
 \|\gamma'\|_{C^\alpha} \leq C(M)
\end{equation*}
 for $\alpha = \frac {p-2} q -1>0$.
As an immediate consequence there is a $\delta=\delta(\a,M)>0$
such that
\begin{equation*}
 |u-v| \leq 2 |\gamma(u)- \gamma(v)|
\end{equation*}
for all $u,v \in \mathbb R / \mathbb Z$ with $|u-v| \leq \delta.$
Let now
\[ S := \inf \sets{\!\rule{0ex}{1em}\abs{\gamma(u)-\gamma(v)}}{u,v \in \R/\Z,\abs{u-v}\geq\delta} \leq \tfrac12. \]
 We will complete the proof by estimating $S$ from below.
 Using the compactness of $\set{u,v \in \R/\Z,\abs{u-v}\geq\delta}$,
 there are $s,t \in \mathbb R / \mathbb Z$ with 
$|s-t| \geq \delta$ and
\begin{equation*}
 |\gamma(s)-\gamma(t)| =S.
\end{equation*}
If now $|s-t| =\delta$ we get
\begin{equation*}
 2 S = 2|\gamma (s) - \gamma(t)| \geq \delta
\end{equation*}
and hence
\begin{equation*}
 |u-v| \leq \tfrac12 \le \frac S\delta \le \frac{\abs{\g(u)-\g(v)}}{\delta(\a,M)}
\end{equation*}
for all $u,v \in \mathbb R / \mathbb Z$ with $|u-v| \geq \delta$.
This proves the proposition in this case.
If on the other hand $|s-t| > \delta$ then we get using the
minimality of $|\gamma(s)-\gamma(t)|$
\begin{equation*}
 \gamma'(s) \perp (\gamma(s) - \gamma(t)) \perp \g'(t).
\end{equation*}
We now set for $\tau \in [-1,1]$
\begin{equation*}
 \gamma_1 (\tau) := \frac 1 S \gamma(s+ S\tau)
 \qquad\text{and}\qquad
 \gamma_2(\tau) := \frac 1 S \gamma(t + S \tau).
\end{equation*}
Since 
\begin{equation*}
 \| \gamma'_i\|_{C^{0,\alpha}} \leq \| \gamma'\|_{C^{0,\alpha}}
 \qquad\br{\le C(M)}
\end{equation*}
we can apply the Lemma~\ref{lem:last} to get
\begin{equation*}
 \E(\gamma_1,\gamma_2) \geq c(\a,M) >0.
\end{equation*}
Together with 
\begin{equation*}
 \E (\gamma_1, \gamma_2) 
 \leq S^{p-q-2} \E (\gamma)
\end{equation*}
this leads to
\begin{equation*}
 S \geq \left( \frac {c(\a,M)}{\E(\gamma)}\right)^{\frac 1 {p-q-2}}
 \ge\br{\frac{c(\a,M)}{M}}^{\frac1{p-q-2}}.
\end{equation*}
Hence,
\[ \abs{u-v} \le\tfrac12\le\frac{\abs{\g(u)-\g(v)}}{2S}
   \le C(M,p,q)\abs{\g(u)-\g(v)} \]
   for all $u,v \in \mathbb R / \mathbb Z$ with $|u-v| \geq \delta$.
\end{proof}

We are now in the position to prove the following mighty

\begin{theorem}[Compactness]\label{thm:sequentiallycompact}
 For each $M< \infty$ the set
 \begin{equation*}
  A_M:=\sets{\gamma \in \Cia(\mathbb R / \mathbb Z, \mathbb R^n)}{ 
  \E(\gamma) \leq M}
 \end{equation*}
 is sequentially compact in $C^1$ up to translations.
\end{theorem}

\begin{proof}
 By Proposition~\ref{prop:conttang}
 there are $C(M)< \infty$ and $\a=\a(p,q)>0$ such that
 \begin{equation*}
  \|\dg\|_{C^\alpha} \leq C(M)
 \end{equation*}
 for all $\g\in A_M$ and hence
 \begin{equation*}
  \|\tilde \g\|_{C^{1,\alpha}} \leq C(M) + 1
 \end{equation*}
 where $\tilde \g (u) := \g(u)-\g(0)$. Furthermore, from
 Proposition~\ref{prop:bi-Lipschitz} we get
 the bi-Lipschitz estimate
 \begin{equation*}
  |u-v| \leq C(M,p,q) |\gamma(u)-\gamma(v)|
 \end{equation*}
 for all $u,v \in \mathbb R / \mathbb Z$.
 
 Let now $\g_n \in A_M$. Then
 \begin{equation*}
  \|\tilde \g_n\|_{C^{1,\alpha}} \leq C(M) + 1
 \end{equation*}
 and hence after passing to suitable subsequence we have
 \begin{equation*}
  \tilde \g_n \rightarrow \g_{0} 
 \end{equation*}
 in $C^1$. Since $\g_n$ was parametrized by arc-length,
 $\g_0$ is still parametrized by arc-length and still
 \begin{equation*}
  |u-v| \leq C(M,p,q) |\gamma_{0}(u) - \gamma_{0}(v)|
 \end{equation*}
 for all $u,v \in \mathbb R / \mathbb Z$.
 So, especially, $\g_{0} \in \Cia(\mathbb R / \mathbb Z, \mathbb R^n)$.
 From lower semi-continuity with respect 
 to $C^1$ convergence we infer
 \begin{equation*}
  \E(\g_0) \leq \liminf_{n\rightarrow \infty} \E(\g_n) \leq M.
 \end{equation*}
 So $\g_0 \in A_M$.
\end{proof}

Let us conclude this section by deriving two simple corollaries 
of this sequential compactness and the lower semi-continuity of the energies
with respect to $C^1$-convergence.

The first one states that
the tangent-point energies~$\E$ are in fact knot energies
as defined in the introduction.
The second one, already stated in the introduction,
ensures that there exist minimizers 
of the energies within every knot class ---
which are then smooth by Theorem~\ref{thm:smooth}.

\begin{proposition}[$\E$ is a strong knot energy {\cite[Prop.~5.1, 5.2]{SM7}}]\label{prop:strong-knot-energy}\hfill\\Let~\eqref{eq:sub-critical-range} hold.\\[-2em]
 \begin{enumerate}
  \item If $\seqn[k]{\g}\subset\Wir$ is a sequence uniformly converging to a non-injective curve $\g_\infty\in C^{0,1}$ parametrized by arc-length then $\E(\g_k)\to\infty$.
  \item For given $E,L>0$ there are only finitely many knot types having a representative with $\E<E$ and $\text{length}=L$.
 \end{enumerate}
\end{proposition}

\begin{proof}
 The first statement immediately follows from the bi-Lipschitz estimate
 in Proposition~\ref{prop:bi-Lipschitz}, as a sequence with bounded energy 
 would be sequentially compact in $\Cia$ and thus cannot uniformly converge 
 to a non-injective curve.   
 
 To show the second statement, let us assume that it was wrong, i.~e.\@
 that there are curves $\seqn[n]{\g}$ of length $L$,
 all belonging to different knot classes,
 with energy less than $E$. Of course we can assume that
 $L=1$. 
 Theorem~\ref{thm:sequentiallycompact} tells us, that
 after suitable translations and going to a subsequence 
 we can assume that
 there is a $\g_0 \in A_M$ such that
 $\g_n \rightarrow \g_0$ in $C^1$.
 As the intersection of every knot class with $C^1$
 is an open set in $C^1$~\cite[Cor.~1.5]{blatt:isot}
 (see~\cite{reiter:isot} for an explicit construction),
 this implies that almost all $\g_n$ belong to the
 same knot class as $\g_0$, which is a contradiction.
\end{proof}

\begin{proof}[Theorem~\ref{thm:existence}]
 Let $\seqn[k]{\g}\in \Cia$ be a minimal sequence for $\E$ in a given knot class $K$,
 i.~e.\@ let
 \begin{equation*}
  \lim_{k\rightarrow \infty} \E(\g_{k}) = \inf_{\Cia\cap K} \E.
 \end{equation*}

 After passing to a subsequence and suitable translations, we hence get by Theorem~\ref{thm:sequentiallycompact} 
 a ~$\g_0\in \Cia$ with $\g_k\to\g_0$ in $C^1$.
 Again by~\cite{blatt:isot,reiter:isot}
 the curve $\g_0$ belongs to the same knot class
 as the elements of the minimal sequence $\seqn[k]{\g}$. 
 The lower semi-continuity of $\E$ furthermore implies that 
 \begin{equation*}
  \inf_{\Cia\cap K} \E\leq \E(\g_0)
  \leq  \lim_{n\rightarrow \infty} \E(\g_n)=\inf_{\Cia\cap K}\E.
 \end{equation*}
 Hence, $\g_0$ is the minimizer we have been searching for.
\end{proof}

By the same reasoning one derives the existence of a global minimizer of~$\E$.

\begin{remark}[Strange range]\label{rem:strange}
 On $p\in[2q+1,q+2)$, $p,q>0$, see the hatched area in Figure~\ref{fig:range},
 we find the strange behavior that there are no
 closed finite-energy $C^{3}$-curves
 while self-intersections, and in particular corners, are not penalized.
 So piecewise linear curves (polygonals) have finite energy.
 
 The latter can be seen by adapting the calculation~\eqref{eq:cross}.
 For the former we recall that a closed
 arc-length parametrized $C^2$-curve must have
 positive curvature $\abs{\g''}$ at some point $u_{0}$ and by continuity
 there are $c,\delta>0$ with
 $\abs{\g''(u_{0}+w)}\ge c>0$
 for all $w\in[-2\delta,2\delta]$.
 As $\g''\perp\dg$ we may lessen $\delta$
 (if necessary) to obtain
 $\abs{\sp{\g''(u+w),\dg(u)}}\le\tfrac12\abs{\g''(u+w)}$
 for all $u\in[u_{0}-\delta,u_{0}+\delta]$, $w\in[-\delta,\delta]$.
 So $\E(\g)$ is bounded below by
 \begin{align*}
  &\int\limits_{u_{0}-\delta}^{u_{0}+\delta}
  \int\limits_{-\delta}^{\delta}
  w^{2q-p}\Bigg(\Bigg|\int\limits_{0}^{1}(1-\th)\g''(u+\th w)\d\th\Bigg|^{2}
  -\Bigg|\sp{\int\limits_{0}^{1}(1-\th)\g''(u+\th w)\d\th,\dg(u)}\Bigg|^{2}\Bigg)^{q/2}\d w\d u \\
  &=\int\limits_{u_{0}-\delta}^{u_{0}+\delta}
  \int\limits_{-\delta}^{\delta}
  w^{2q-p}\Bigg(\tfrac12\iint\limits_{[0,1]^{2}}
  (1-\th_{1})(1-\th_{2})
  \Bigg[
  \abs{\g''(u+\th_{1}w)}^{2}
  +\abs{\g''(u+\th_{2}w)}^{2}-{} \\
  &\qquad\qquad\qquad\qquad
  {}-\abs{\g''(u+\th_{1}w)-\g''(u+\th_{2}w)}^{2}
  -\abs{\sp{\g''(u+\th_{1}w),\dg(u)}}^{2} -{} \\
  &\qquad\qquad\qquad\qquad
  {}-\abs{\sp{\g''(u+\th_{2}w),\dg(u)}}^{2}
  +\abs{\sp{\g''(u+\th_{1}w)-\g''(u+\th_{2}w),\dg(u)}}^{2}\Bigg] \\
  &\qquad\qquad\qquad\qquad
  {}\d\th_{1}\d\th_{2}\Bigg)^{q/2}\d w\d u.
 \end{align*}
 Lessening $\delta>0$ once more, the term
 $\abs{\g''(u+\th_{1}w)-\g''(u+\th_{2}w)}^{2}
 \le w^{2}\lnorm[\infty]{\g'''}^{2}$ can be made
 so small that the square bracket is $\ge\tilde c>0$.
 This gives $\E(\g)=\infty$.
\end{remark}
\section{First variation}\label{sect:frech}

Let turn to the proof of Theorem~\ref{thm:first-var}. In contrast to
the investigation of O'Hara's energies~\cite{blatt-reiter2},
we do not need to cut off the singular part in the energies. Instead, 
a straightforward calculation of the first variation using Lebesgue's 
theorem of dominated convergence will prove that
\begin{equation*}
 \delta \E(\g,h) := \lim_{\tau \to 0} \frac {\E (\gamma+\tau h) - \E(\gamma)}{\tau}
\end{equation*}
exists.

For $\g \in \Wia (\mathbb R / \mathbb Z, \mathbb R^n)$
and $h \in \W[(p-1)/q,q]$ let
\[\gt:=\g+\tau h \qquad \text{for any } \tau\in [-\tau_0,\tau_0] \]
where $\tau_0 \in (0,1)$
is so small that
\begin{equation}\label{eq:tau0}
 \abs\dgt\ge\tfrac12 \qquad\text{on }\R/\Z
\end{equation}
and each curve $\g+\tau h$, $\tau\in[-\tau_{0},\tau_{0}]$, is still injective.
Then, recalling~\eqref{eq:short-notation},
\begin{equation*}
 \frac {\E (\gamma+\tau h) - \E(\gamma)}{\tau}
 = \Wint I_\tau(u,w)\d w\d u
\end{equation*}
where
\begin{equation*}
  I_\tau(u,w) := \frac 1 \tau \left( 
   \frac{\abs{\PPT\br{\T\g_\tau}}^q}{\abs{\T\g_\tau}^p} \abs{\dg_\tau(u+w)}\abs{\dg_\tau(u)} 
   -  \frac{\abs{\PP\br{\T\g}}^q}{\abs{\T\g}^p} \abs{\dg(u+w)}\abs{\dg(u)}
  \right).
\end{equation*}

To calculate the pointwise limit of $I_\tau (u,w)$ as $\tau \to 0$,
we observe using $|\dg|\equiv 1$ that
\begin{eqnarray*}
  \dd \Big|_{\tau=0} \abs{\dg_\tau(u)} 
  = \sp{\dg(u),\dh(u)}, \\
  \dd \Bigg|_{\tau =0} \fracabs{\dgt(u)}
  =\PP \dh(u),
 \end{eqnarray*}
and
\begin{equation*}
 \PP \left( \dd\Big|_{\tau=0} { \PPT v }\right)
 = - \sp{v,\dg_\tau (u)} \PP \dh(u),
\end{equation*}
which gives
\begin{align*}
 &{\sp{\PP\T\g ,\T h} - \sp{\T\g,\dg(u)}\sp{\PP \T\g,\dh(u)}} \\
 & = {\sp{\PP\br{\T\g-w\dg(u)},\T h-\sp{\T\g,\dg(u)}\dh(u)}}.
\end{align*}
Hence,
\begin{align*}
 &\lim_{\tau \to 0} I_{\tau}(u,w) \\
 &\qquad= q\frac{\sp{\PP\br{\T\g-w\dg(u)},\T h-\sp{\T\g,\dg(u)}\dh(u)}}{\abs{\T\g}^p} \abs{\PP\br{\T\g-w\dg(u)}}^{q-2} \\
 &\qquad\quad{} - p\frac{\abs{\PP\br{\T\g}}^q\sp{\T\g,\T h}}{\abs{\T\g}^{p+2}} \\
 &\qquad\quad{} + \frac{\abs{\PP\br{\T\g}}^q}{\abs{\T\g}^{p}}\br{\sp{\dg(u),\dh(u)}+\sp{\dg(u+w),\dh(u+w)}}.
 \end{align*}
 We decompose
\begin{align*}
 &I_\tau(u,w) \\
 &=  \frac 1 \tau\Bigg( \frac{\abs{\PP[\dgt(u)]\br{\T\gt}}^q-\abs{\PP\br{\T\g}}^q}{\abs{\T\gt}^p} \abs{\dgt(u+w)}\abs{\dgt(u)} \\
 &\qquad\qquad{}+ \abs{\PP\br{\T\g}}^q\br{\frac1{\abs{\T\gt}^p}-\frac1{\abs{\T\g}^p}} \abs{\dgt(u+w)}\abs{\dgt(u)} \\
 &\qquad\qquad{}+ \frac{\abs{\PP\br{\T\g}}^q}{\abs{\T\g}^p}\br{\abs{\dgt(u+w)}\abs{\dgt(u)}-\abs{\dg(u+w)}\abs{\dg(u)}} \Bigg)
 \\
 &\qquad=: \Fe[1] + \Fe[2] + \Fe[3].
\end{align*}
We will give uniform majorants for these three terms.
In order to treat the first term we first consider
\begin{align*}
 \P[\dgt(u)]a-\P a &=\sp{a,\dgt}\dgt\br{\frac1{\abs{\dgt}^2}-1}+\sp{a,\dgt}\dgt-\sp{a,\dg}\dg \\
 &= \tau\br{-\sp{a,\dgt}\dgt\frac{2\sp{\dg,\dh}+\tau\abs{\dh}^2}{\abs{\dgt}^2} + \sp{a,\dh}\dgt+\sp{a,\dg}\dh}
\end{align*}
which for $a:=\T\g-w\dg(u)$ gives
\begin{align*}
 &\PP[\dgt(u)]\br{\T\gt}-\PP\br{\T\g} \\
 &=\PP[\dgt(u)]\br{\T\gt-w\dgt(u)}-\PP\br{\T\g-w\dg(u)} \\
 &=\PP[\dgt(u)]\br{\T\g-w\dg(u)}+\tau\PP[\dgt(u)]\br{\T h-w\dh(u)}-\PP\br{\T\g-w\dg(u)} \\
 &=\P\br{\T\g-w\dg(u)} - \P[\dgt(u)]\br{\T\g-w\dg(u)}+\tau\PP[\dgt(u)]\br{\T h-w\dh(u)} \\
 &= \tau\br{\sp{a,\dgt}\dgt\frac{2\sp{\dg,\dh}+\tau\abs{\dh}^2}{\abs{\dgt}^2} - \sp{a,\dh}\dgt-\sp{a,\dg}\dh+\PP[\dgt(u)]\br{\T h-w\dh(u)}}.
\end{align*}
Recalling~\eqref{eq:tau0} and $\abs\dg\equiv1$,
we hence get a constant $C$ depending on $\norm\dh_{L^\infty}$ and $\tau_0$
such that
\begin{equation*}
 \left|\PP[\dgt(u)]\br{\T\gt}-\PP\br{\T\g} \right|
 \leq C |\tau| \left( |\T\g-w\dg(u)| + |\T h-w h'(u)|\right).
\end{equation*}

By $\abs{a^{q}-b^{q}}\le q\abs{a-b}\max\br{a^{q-1},b^{q-1}}$
for $a,b\ge0$, $q>1$ we deduce for $C=C\br{\lnorm[\infty]\dh,q,\tau_{0}}>0$
\begin{align*}
 &\abs{\abs{\PP[\dgt(u)]\br{\T\gt}}^{q}-\abs{\PP\br{\T\g}}^{q}} \\
 &\le C\abs\tau\br{\abs{\T\g-w\dg(u)}+\abs{\T h-w\dh(u)}}\br{\abs{\T\gt-w\dgt(u)}^{q-1}+\abs{\T\g-w\dg(u)}^{q-1}} \\
 &\le C\abs\tau\br{\abs{\T\g-w\dg(u)}+\abs{\T h-w\dh(u)}}\br{\abs{\T\g-w\dg(u)}^{q-1}+\abs{\T h-w\dh(u)}^{q-1}} \\
 &\le C\abs\tau\br{\abs{\T\g-w\dg(u)}^q+\abs{\T h-w\dh(u)}^q} \\
 &\le C\abs\tau\abs w^{p}\frac{\abs{\int_{0}^{1}\br{\dg(u+\theta w)-\dg(u)}\d\theta}^q+\abs{\int_{0}^{1}\br{\dh(u+\theta w)-\dh(u)}\d\theta}^q}{\abs w^{p-q}}
\end{align*}
and hence, by Equations~\eqref{eq:bilip}, \eqref{eq:tau0},
\begin{equation*}
 \abs{\Fe[1]} \leq C_{\g}\frac{\abs{\int_{0}^{1}\br{\dg(u+\theta w)-\dg(u)}\d\theta}^q+\abs{\int_{0}^{1}\br{\dh(u+\theta w)-\dh(u)}\d\theta}^q}{\abs w^{p-q}}.
\end{equation*}

Applying Jensen's inequality, one sees
\begin{align*}
 &\Wint\abs{F_{1}}\d w\d u \\
 &\le C_{\g}\int_{\mathbb R / \mathbb Z} \int_{-1/2}^{1/2} \int_{0}^{1}
 \frac{\abs{\dg(u+\theta w)-\dg(u)}^q+\abs{\dh(u+\theta w)-\dh(u)}^q}{\abs w^{p-q}} \d\theta\d w\d u\\
 &\le C_{\g}\int_{\mathbb R / \mathbb Z} \int_{-1/2}^{1/2}
 \frac{\abs{\dg(u+w)-\dg(u)}^q+\abs{\dh(u+w)-\dh(u)}^q}{\abs w^{p-q}} \d w\d u,
\end{align*}
so we have found an $L^1$-majorant for $\Fe[1]$.

The same conclusions lead to a majorant for the remaining terms to which we pass now.
Using arc-length parametrization and~\eqref{eq:tau0} together with $\abs{a^{-p}-b^{-p}} \le C_{\mu,p}\abs{a-b}$ for $a,b\ge\mu>0$,
we compute
\begin{align*}
 \abs{\Fe[2]}
 &\le \frac C \tau  \abs{\PP\br{\T\g}}^q\abs{\frac1{\abs{\T\gt}^p}-\frac1{\abs{\T\g}^p}} \\
 &\le \frac C \tau  \abs{\PP\br{\T\g-w\dg(u)}}^q\abs{w}^{-p}\abs{\br{\frac{\abs w}{\abs{\T\gt}}}^p-\br{\frac{\abs w}{\abs{\T\g}}}^p} \\
 &\le \frac C \tau  \abs{\int_0^1\br{\dg(u+\theta w)-\dg(u)}\d\theta}^q \abs{w}^{q-p}\abs{\frac{\abs{\T\gt}}{\abs w}-\frac{\abs{\T\g}}{\abs w}} \\
 &\le C \norm{\dh}_{L^{\infty}}  \int_0^1\abs{\dg(u+\theta w)-\dg(u)}^q\d\theta \abs{w}^{q-p}
\end{align*}
and get using a simple substitution
\begin{equation*}
\int_{\mathbb R / \mathbb Z} \int_{-1/2}^{1/2} \abs{F_{2}} \d w \d u
\le C \Wint \frac{\abs{\dg(u+w)-\dg(u)}^q}{\abs w^{p-q}} \d w\d u.
\end{equation*}
Finally
\[ \abs\dgt-\abs\dg = \abs\dgt-1 = \frac{\abs\dgt^2-1}{\abs\dgt+1} = \tau \frac{2\sp{\dg,\dh}+\tau\abs\dh^2}{\abs\dgt+1} \]
permits to proceed as in the proof of Proposition~\ref{prop:energ-space}. We arrive at
\[ \Wint\abs{\Fe[3]}\d w\d u\le C \Wint\frac{\abs{\dg(u+w)-\dg(u)}^q}{\abs w^{p-q}} \d w\d u. \]

 Hence, the functions $I_\tau$ have a uniform $L^1$ majorant. Thus
 Lebesgue's theorem of dominant convergence implies
 \begin{align*}
 &\delta \E( \gamma, h)  \\
  &= \Wint \Bigg\{q\frac{\sp{\PP\br{\T\g-w\dg(u)},\T h-\sp{\T\g,\dg(u)}\dh(u)}}{\abs{\T\g}^p} \abs{\PP\br{\T\g-w\dg(u)}}^{q-2} \\
 &\qquad\qquad\qquad{} - p\frac{\abs{\PP\br{\T\g}}^q\sp{\T\g,\T h}}{\abs{\T\g}^{p+2}} \\
 &\qquad\qquad\qquad{} + \frac{\abs{\PP\br{\T\g}}^q}{\abs{\T\g}^{p}}\br{\sp{\dg(u),\dh(u)}+\sp{\dg(u+w),\dh(u+w)}} \Bigg\}.
 \end{align*}

\begin{remark}[Continuous differentiability]
Estimating more carefully, one can even show that $\delta\E(\cdot,\cdot)$ 
is continuous on $\Wir\times\W[(p-1)/q,q]$,
i.~e.\@
\begin{equation}
 \E\in C^1\br{\Wir\rzd} \qquad\text{for }p\in(q+2,2q+1).
\end{equation}
In contrast to O'Hara's energies, even an explicit formula of the first variation
can be given at \emph{arbitrary} $\gamma \in \Wir[(p-1)/q,q]$ in the direction
$h \in \W[(p-1)/q,q]$. In fact, we have
\begin{align*}
 &\delta\E(\g,h)\\
 &=\Wint\Bigg\{q\frac{\sp{\PP \T\g,\T h}-\sp{\T\g,\frac{\dg(u)}{|\dg(u)|^2}}
 \sp{\PP\T\g,\dh(u)}}{\abs{\T\g}^p} \abs{\PP\T\g}^{q-2} \\
 &\qquad\qquad\qquad{} - p\frac{\abs{\PP\br{\T\g}}^q\sp{\T\g,\T h}}{\abs{\T\g}^{p+2}} \\
 &\qquad\qquad\qquad{} + \frac{\abs{\PP\br{\T\g}}^q}{\abs{\T\g}^{p}}\br{\sp{\frac{\dg(u)}{|\dg(u)|^2},\dh(u)}+\sp{\frac{\dg(u+w)}{|\dg(u+w)|^{2}},\dh(u+w)}} \Bigg\} 
 \\
 & \qquad\qquad\quad{}\cdot |\dg(u)| \ |\dg(u+w)|\d u \d w.
\end{align*}
\end{remark}
\section{Bootstrapping}\label{sect:bootstrap}

\renewcommand{\E}[1][p,2]{\ensuremath{\mathrm{TP}^{(#1)}}}
\renewcommand{\W}[1][(p-3)/2,2]{\ensuremath{W^{\scriptstyle #1}}}
\renewcommand{\Wia}[1][(p-3)/2,2]{\ensuremath{W_{\mathrm{ia}}^{#1}}}
\renewcommand{\Wir}[1][(p-3)/2,2]{\ensuremath{W_{\mathrm{ir}}^{#1}}}

In this section we consider the non-degenerate sub-critical case
\begin{equation}\tag{\ref{eq:regularity-range}}
 p\in(4,5), \qquad q=2
\end{equation}
which corresponds to the yellow line in Figure~\ref{fig:range}.

In order to start a bootstrapping process, we have to
rearrange the Euler-Lagrange Equation for $\E$
exhibiting a gap of regularity between suitable terms.
To this end, we decompose
$\delta\E$ into the sum of a bilinear elliptic term~$\Q$
and a remainder term~$\Re[]$ of lower order. The former is defined via
\begin{equation*}
 \Q(f,g) := \Wint \frac{\sp{\T f-wf'(u),\T g-wg'(u)}}{\abs w^{p}}\d w\d u
\end{equation*}
for $f,g \in \W[(p-1)/2,2](\mathbb R /\mathbb Z, \mathbb R^n )$.

The operator $\Q$ is characterized by the following

\begin{proposition}[Bilinear elliptic operator]\label{prop:Q}
 The functional $\Q$ is bilinear on $\br{\W[(p-1)/2,2]}^{2}$, more precisely
 \[ \Q(f,g) = \sum_{k\in\Z} \rho_k\sp{\hat f_k,\hat g_k}_{\C^d} \qquad \text{where } \rho_k = c\abs k^{p-1} + o\br{\abs k^{p-1}} \text{ as }\abs k\nearrow\infty \]
 and $c>0$.
 Here $\hat\cdot_k$ denotes the $k$-th Fourier coefficient
 \[
  \hat f_k := \int_{0}^1 f(x) e^{-2\pi ikx } \d x.
 \]
\end{proposition}

\begin{proof}
  Since any $L^2$ function is uniquely determined by its Fourier series,
  we obtain for $f,g\in\W[(p-1)/2,2]$
  \begin{equation}
      \Q(f,g) =
      2 \int_0^{1/2} \sum_{k\in\Z}
      \sp{\hat f_k,\hat g_k}_{\C^d}
      \frac{\digamma(2\pi kw)}{w^{p}} \d w,\label{eq:qk}
  \end{equation}
  where $\digamma(x) := 2-2\cos x-2x\sin x+x^2$, which turns out to be even and monotone increasing on
  $\set{x\ge0}$, for $\dot\digamma(x) = 2x(1-\cos x) \ge 0$.
  Now $\digamma(0)=0$ implies that $\digamma$
  is non-negative on~$\R$, so
  $\Q(f,g) = \sum_{k\in\Z} \rho_k\sp{\hat f_k,\hat g_k}_{\C^d}$ where
  \begin{equation}
    \rho_k := 2 \int_0^{1/2}\frac{\digamma(2\pi kw)}{w^{p}} \d w = 2\abs{2\pi k}^{p-1} \int_0^{\abs k\pi}\frac{\digamma(x)}{x^{p}} \d x
    \label{eq:Q=q}
  \end{equation}
  is finite by $\digamma(x)=O(x^{4})$ as $x\to0$.
  Obviously, $\rho_k$ is monotone increasing in $\abs k$, so $\lim_{\abs k\to\infty}\rho_k\abs k^{-p+1} \in (0,\infty)$ by $\digamma(x)=O(x^{2})$ as $\abs x\to\infty$.
\end{proof}

Using~\eqref{eq:dE}, we now consider the remainder term
\begin{align}
 &\Re(\g,h) := \delta\E(\g,h)-2\Qe(\g,h) \label{eq:Re}\\
 &= \Wint \Bigg\{2\sp{\PP\br{\T\g-w\dg},\T h-\sp{\T\g,\dg}\dh} \br{\frac1{\abs{\T\g}^p}-\frac1{\abs w^p}} \nonumber\\
 &\qquad\qquad\qquad{} + 2 \frac{\sp{\PP\br{\T\g-w\dg},\T h-\sp{\T\g,\dg}\dh}-\sp{\T\g-w\dg,\T h-w\dh}}{\abs w^p} \nonumber\\
 &\qquad\qquad\qquad{} - p \frac{\abs{\PP\br{\T\g}}^2\sp{\T\g,\T h}}{\abs{\T\g}^{p+2}} \nonumber\\
 &\qquad\qquad\qquad{} + \frac{\abs{\PP\br{\T\g}}^2}{\abs{\T\g}^{p}}\br{\sp{\dg(u),\dh(u)}+\sp{\dg(u+w),\dh(u+w)}} \Bigg\} \d w\d u\nonumber\\
 &=: \Re[1] + \Re[2] + \Re[3] + \Re[4].\nonumber
\end{align}
Interestingly, all these terms have the same structure
so we can treat them simultaneously.
As in analysis the exact form of a multilinear mapping $\br{\R^{n}}^{N}\to\R$ 
does not matter at all,
let us introduce the $\oast$ notation which represents any sort of
these operators,
e.~g., $\sp{\br{a\otimes b}c,d}=a\oast b\oast c\oast d$
for $a,b,c,d\in\R^{n}$.

\begin{lemma}[Structure of the remainder term]\label{lem:struc-rem}
 The term $\Re(\g,h)$ is a (finite) sum of expressions of type
 \begin{equation*}
  \Wint \idotsint_{[0,1]^{K}}
  \gp(u,w) \oast\dh(u+s_{K}w) \d\theta_{1}\cdots\d\theta_{K}\d w\d u
 \end{equation*}
 where
 \begin{equation*}
  \gp(u,w) := \G\br{\abs{\frac{\T\g}w}}\frac{\abs{\dg(u+s_{1}w)-\dg(u+s_{2}w)}^{2}}{\abs w^{p-2}}\br{\bigoast_{i=3}^{K-1}\dg(u+s_{i}w)},
  \end{equation*}
  $\G$ is some analytic function defined on $[c,\infty)$,
  and $s_{i}\in\set{0,\theta_{i}}$ for all $i=1,\dots,K$.
\end{lemma}

\begin{proof}
 We begin with $\Re[1]$. Using $\PP=\mathbbm1-\g\otimes\g$
 where $\mathbbm1$ denotes the unit matrix, we obtain
 \begin{align*}
  &w^{-2}\sp{\PP\br{\T\g-w\dg},\T h-\sp{\T\g,\dg}\dh} \\
  &= \iiint_{[0,1]^{3}} \sp{(\mathbbm1-\dg\otimes\dg)\br{\dg(u+\theta_{1}w)-\dg},
     \dh(u+\theta_{3}w)-\sp{\dg(u+\theta_{2}w),\dg}\dh} \d\theta_{1}\d\theta_{2}\d\theta_{3}.
 \end{align*}
 As $\abs\dg=1$, we may add $\sp{\dg,\dg}$ before $\mathbbm1$ and
 $\dh(u+\theta_{3}w)$. Expanding the three differences,
 we obtain eight terms all of which have the form
 \[ \iiint_{[0,1]^{3}} \dg\oast\dg\oast\dg(u+s_{1}w)\oast\dg(u+s_{2}w)\oast\dg\oast\dh(u+s_{3}w)\d\theta_{1}\d\theta_{2}\d\theta_{3} \]
 where $s_{i}\in\set{0,\theta_{i}}$, $i=1,2,3$.
 For the remaining factor we set $\G(z):=\frac{z^{-p}-1}{z^{2}-1}$ and compute,
 using $\abs\dg\equiv1$ and $\sp{a,b}-1=-\frac12\abs{a-b}^{2}$
 for $\abs a=\abs b=1$,
 \begin{align*}
 &w^{2}\br{\frac1{\abs{\T\g}^p}-\frac1{\abs w^p}} \\
 &=\abs w^{{2-p}} \br{\abs{\frac{\T\g}w}^{-p}-1} \\
 &=\G\br{\abs{\frac{\T\g}w}} \frac{\abs{\frac{\T\g}w}^{2}-1}{\abs w^{p-2}} \\
 &=\G\br{\abs{\frac{\T\g}w}} \frac{\abs{\int_{0}^{1}\dg(u+\theta w)\d\theta}^{2}-1}{\abs w^{p-2}} \\
 &=-\tfrac12\G\br{\abs{\frac{\T\g}w}} \frac{\iint_{[0,1]^{2}}\abs{\dg(u+\theta_{1} w)-\dg(u+\theta_{2} w)}^{2}\d\theta_{1}\d\theta_{2}}{\abs w^{p-2}}. \\
 \end{align*}
 So $\Re[1]$ has the desired form. The nominator of $\Re[2]$ reads
 \begin{align*}
  & \sp{\PP\br{\T\g-w\dg},\T h-\sp{\T\g,\dg}\dh}-\sp{\T\g-w\dg,\T h-w\dh} \label{eq:R2}\\
  &= \sp{\T\g-w\dg,w\dh-\sp{\T\g,\dg}\dh} - \sp{\T\g-w\dg,\dg}\sp{\T h-\sp{\T\g,\dg}\dh,\dg} \nonumber\\
  &= \sp{\T\g-w\dg,\sp{w\dg-\T\g,\dg}\dh} - \sp{\T\g-w\dg,\dg}\sp{\T h-\sp{\T\g,\dg}\dh,\dg} \nonumber\\
  &= -\sp{\T\g-w\dg,\dg}\br{\sp{\T\g-w\dg,\dh} + \sp{\T h-\sp{\T\g,\dg}\dh,\dg}} \nonumber\\
  &= -w^2\int_0^1\br{\sp{\dg(u+\theta_1 w),\dg(u)}-1}\d\theta_1 \cdot \nonumber\\
  &\qquad\qquad{}\cdot\iint_{[0,1]^{2}}\br{\sp{\dg(u+\theta_2 w)-\dg,\dh} + \sp{\dh(u+\theta_3w)-\sp{\dg(u+\theta_2w),\dg}\dh,\dg}}\d\theta_2\d\theta_{3} \nonumber\\
  &= \tfrac12w^2\int_0^1\abs{\dg(u+\theta_1 w)-\dg(u)}^2\d\theta_1 \cdot \iint_{[0,1]^{2}}\br{\cdots}\d\theta_2\d\theta_{3}
 \end{align*}
 where $\br\cdots$ is a sum of terms of type
 $\dg(u+s_{2}w)\oast\dg\oast\dg\oast\dh(u+s_{3}w)$
 with $s_{i}\in\set{0,\theta_{i}}$, $i=2,3$.
 For $\Re[3]$ we set $\G(z):=\abs z^{-p-2}$ and obtain
 \begin{align*}
 &\frac{\abs{\PP\br{\T\g}}^2\sp{\T\g,\T h}}{\abs{\T\g}^{p+2}}
 =\frac{\sp{\PP\br{\T\g},\T\g}}{\abs{\T\g}^{p+2}}\sp{\T\g,\T h} \\
 &=\G\br{\abs{\frac{\T\g}{w}}}\iiiint_{[0,1]^{4}}\frac{\sp{\dg(u+\theta_{1}w)-\sp{\dg(u+\theta_{1}w),\dg}\dg,{\dg(u+\theta_{2}w)}}}{\abs w^{p-2}}\cdot{}\\
 &\qquad\qquad\qquad\qquad\qquad\qquad{}\cdot\sp{\dg(u+\theta_{3}w),\dh(u+\theta_{4}w)}\d\theta_{1}\d\theta_{2}\d\theta_{3}\d\theta_{4}.
 \end{align*}
 The nominator reads
 \begin{align*}
 &\sp{\dg(u+\theta_{1}w)-\sp{\dg(u+\theta_{1}w),\dg}\dg,{\dg(u+\theta_{2}w)}} \\
 &=\sp{\dg(u+\theta_{1}w),\dg(u+\theta_{2}w)}-1
   +1-\sp{\dg(u+\theta_{1}w),\dg}
   +\sp{\dg(u+\theta_{1}w),\dg}\br{1-\sp{\dg(u+\theta_{2}w),\dg}} \\
 &=-\tfrac12\abs{\dg(u+\theta_{1}w)-\dg(u+\theta_{2}w)}^{2}
   +\tfrac12\abs{\dg(u+\theta_{1}w)-\dg}^{2}
   +\tfrac12\sp{\dg(u+\theta_{1}w),\dg}\abs{\dg(u+\theta_{2}w)-\dg}^{2}.
 \end{align*}
 Finally, $\Re[4]$ is treated similarly to $\Re[3]$.
\end{proof}

Our next task is to show that
 $\Re[]$
is in fact a lower-order term. More precisely, we have

\begin{proposition}[Regularity of the remainder term]\label{prop:reg-rem}
 If $\g\in\Wia[(p-1)/2+\s,2]\rzd$ for some $\s\ge0$ then
 $\Re[](\g,\cdot)\in\br{\W[3/2-\s+\eps,2]}^{*}$ for any $\eps>0$.
\end{proposition}

This statement together with Proposition~\ref{prop:Q} immediately leads to
the proof of the regularity theorem which is deferred to the end of
this section.

To prove Proposition~\ref{prop:reg-rem}, we first note that,
by partial integration, the terms of $\Re(\g,h)$ may be transformed into
\begin{align*}
 &\idotsint_{[0,1]^{K}}\int_{-1/2}^{1/2}\int_{\R/\Z}
  \br{(-\Delta)^{\tilde\s/2}\gp(\cdot,w)}(u) \oast\br{(-\Delta)^{-\tilde\s/2}\dh}(u+s_{K}w)
  \d u\d w \d\theta_{1}\cdots\d\theta_{K} \\
 &\le \idotsint_{[0,1]^{K}}
 \int_{-1/2}^{1/2}\norm{\gp (\cdot, w)}_{\W[\tilde\s,1]\rzd}\d w
 \d\theta_{1}\cdots\d\theta_{K-1} \norm{(-\Delta)^{-\tilde\s/2}\dh}_{L^{\infty}\rzd} \\
 &\le \idotsint_{[0,1]^{K}}
 \int_{-1/2}^{1/2}\norm{\gp (\cdot, w)}_{\W[\tilde\s,1]\rzd}\d w
 \d\theta_{1}\cdots\d\theta_{K-1}\norm{h}_{\W[3/2+\eps/2-\tilde\s,2]\rzd},
\end{align*}
 where $\tilde\s\in\R$, $\eps>0$ can be chosen arbitrarily, and $(-\Delta)^{\tilde\s/2}$ denotes the fractional
 Laplacian.
 We let $\tilde\s:=0$ if $\s=0$ and $\tilde\s:=\s-\frac\eps2$ otherwise.
 Now the claim directly follows from the succeeding auxiliary result.

\begin{lemma}[Regularity of the remainder integrand]\label{lem:reg-int}
 Let $\g\in\Wia[(p-1)/2+\s,2]$.
 \begin{itemize}
 \item If $\s=0$ then $\gp\in L^{1}(\R/\Z\times(-\frac12,\frac12),\R^{n})$ and
 \item if $\s>0$ then $(w \mapsto \gp(\cdot,w))\in L^{1}((-\frac12,\frac12),\W[\tilde\s,1](\R/\Z, \R^{n}))$ for any $\tilde\s<\s$.
 \end{itemize}
 The respective norms are bounded independently of $s_{1},\dots,s_{K}$.
\end{lemma}

\begin{proof}
 Note that, by~\eqref{eq:bilip}, the argument of $\G$ is compact and bounded away from zero.
 Using arc-length parametrization, we immediately obtain
 \[ \abs{\gp(u,w)} \le C\frac{\abs{\dg(u+s_{1}w)-\dg(u+s_{2}w)}^{2}}{\abs w^{p-2}} \]
 which gives $\Wint\abs{\gp(u,w)}\d w\d u
 \le C\seminorm\dg_{\W[(p-3)/2,2]}$. To prove the second claim,
 we will derive a suitable bound on $\norm{\gp(\cdot,w)}_{\W[\tilde\s,r]}$ for some $r>1$.
 We choose $q_{1},\dots,q_{K-1}$, which will be determined more precisely later on,
 such that
 \begin{equation*}
   \sum_{i=1}^{K-1} \frac 1 {q_i} = \frac 1 r.
 \end{equation*}
 Lemma~\ref{lem:product} then leads to 
 \[ \norm{\gp(\cdot,w)}_{\W[\tilde\s,r]} \le C\norm{\G\br{\abs{\frac{\T\g}w}}}_{\W[\tilde\s,q_{1}]}\frac{\norm{\dg(\cdot+s_{1}w)-\dg(\cdot+s_{2}w)}_{\W[\tilde\s,2q_{2}]}^{2}}{\abs w^{p-2}}{\prod_{i=3}^{K-1}\norm\dg_{\W[\tilde\s,q_{i}]}}. \]
 
 For the second factor, we now choose $q_{2}>r$ so small that
 $\W[\s,2]$ embeds into $\W[\tilde\s,2q_{2}]$.
 To this end, we set $\frac1r:=1-(\s-\tilde\s)$ and $\frac1{q_{2}}:=1-2(\s-\tilde\s)$.
 and $q_{i}:=\frac{K-2}{\s-\tilde\s}$ for $i=1,3,4, \ldots, K-1$.
 
 Then for the first factor we apply Lemma~\ref{lem:chain}.
 Recall that $\G$ is analytic and its argument is bounded below away from zero
 and above by~$1$.
 We infer
 \begin{equation*}
  \norm{\G\br{\abs{\frac{\g(\cdot+w)-\g(\cdot)}w}}}_{\W[\tilde\s,q_{1}]}
  \leq C \norm\dg_{\W[ \tilde \sigma ,q_{1}]}.
 \end{equation*}
 The Sobolev embedding gives
 \[ \norm{\dg}_{\W[ \tilde \sigma ,q_{i}]} \le C \norm\g_{\W[(p-1)/2+\s,2]} \le  C \qquad \text{for } i=1,3,4,\dots,K-1. \]
 Together this leads to
 \[ \norm{\gp(\cdot,w)}_{\W[\tilde\s,r]} \le C\frac{\norm{\dg(\cdot+s_{1}w)-\dg(\cdot+s_{2}w)}_{\W[\s,2]}^{2}}{\abs w^{p-2}} \]
 and finally
 \[ \int_{-1/2}^{1/2} \norm{\gp(\cdot,w)}_{\W[\tilde\s,r]} \d w
    \le C\norm\g_{\W[(p-1)/2+\s,2]}^{2}. \]
\end{proof}

\begin{proof}[Theorem~\ref{thm:smooth}]
 We arrive at the Euler-Lagrange Equation
 \begin{equation}\label{eq:elg}
  \delta\E(\g,h)+\lambda\sp{\dg,\dh}_{L^2} = 0
 \end{equation}
 for any $h\in C^\infty(\R/\Z)$
 where $\lambda\in\R$ is a Lagrange parameter stemming from the side condition (fixed length).
 Using~\eqref{eq:Re} this reads
 \begin{equation}\label{eq:elg2}
  2\Q(\g,h) + \lambda\sp{\dg,\dh}_{L^2} + \Re[](\g,h) = 0.
 \end{equation}
 Since first variation of the length functional satisfies
 \begin{equation*}
  \sp{\dg,\dh}_{L^2} = \sum_{k\in \mathbb Z} |2\pi k|^2 \sp{\hat\g_{k},\hat h_{k}}_{\C^{d}},
 \end{equation*}
 we get using  Proposition~\ref{prop:Q} that there is a $c>0$ such that
 \begin{equation}\label{eq:LinearTerm}
 2\Q(\g,h) + \lambda\sp{\dg,\dh}_{L^2} = \sum_{k\in\Z} \tilde\rho_k\sp{\hat \g_k,\hat h_k}_{\C^d}
 \end{equation}
 where
 \[ \tilde\rho_k = c\abs k^{p-1} + o\br{\abs k^{p-1}}
    \qquad\text{as }\abs k\nearrow\infty. \]
 Assuming that $\g\in\Wia[(p-1)/2+\s,2]$ for some $\s\ge0$,
 we infer
 \[ 2\Q(\g,\cdot) + \lambda\sp{\dg,\cdot'}_{L^2} \in \br{\W[3/2-\s+\eps,2]}^{*} \]
 from applying Proposition~\ref{prop:reg-rem} to~\eqref{eq:elg2}.
 Equation~\eqref{eq:LinearTerm} implies
 \[ \br{\rho_{k}\abs k^{-3/2+\sigma-\eps}\hat\g_{k}}_{k\in\Z} \in \ell^{2}. \]
 Recalling that $\rho_{k}\abs k^{{-p+1}}$ converges to a positive constant as $\abs k\nearrow\infty$, we are led to
 \[ \g\in\W[\textstyle\frac{p-1}2+\s+\frac{p-4}2-\eps]. \]
 Choosing $\eps:=\frac{p-4}4>0$, we gain a positive amount of regularity
 that does not depend on~$\s$. So by a simple interation we get $\g \in W^{s,2}$
  for all $s \geq 0$.
 \end{proof}

\begin{appendix}
 \section{Product and chain rule}\label{sect:fractional}

As in~\cite{blatt-reiter2}, we make use of the following results which
we briefly state for the readers' convenience.

\begin{lemma}[Product rule]\label{lem:product}
  Let $q_1,\dots,q_{k} \in (1,\infty)$ with $\sum_{i=1}^{k}\frac1{q_{k}}=\frac1r\in(1,\infty)$ and 
  $s > 0$.
  Then, for $f_{i}\in\W[s,q_{i}]\rzd$, $i=1,\dots,k$,
 \begin{equation*}
  \norm{\prod_{i=1}^{k} f_{i}}_{\W[s,r]} \leq C_{k,s} \prod_{i=1}^{k}\norm{f_{i}}_{\W[s,q_{i}]}.
 \end{equation*}
\end{lemma}

We also refer to Runst and Sickel~{\cite[Lem.~5.3.7/1~(i)]{RS}}. ---
For the following statement, one mainly has to treat $\norm{(D^k\psi)\circ f}_{\W[\s,p]}$ for $k\in\N\cup\set0$ and $\s\in(0,1)$
which is e.~g.\@ covered by~\cite[Thm.~5.3.6/1~(i)]{RS}.

\begin{lemma}[Chain rule]\label{lem:chain}
 Let $f\in\W[s,p](\mathbb R / \mathbb Z, \mathbb R^n)$, $s > 0$, $p \in (1,\infty)$.
 If $\psi\in C^\infty(\R)$ is globally Lipschitz continuous and $\psi$ and all its derivatives vanish at~$0$ then $\psi\circ f\in\W[s,p]$ and
 \[ \|\psi \circ f\|_{\W[s,p]} \le C\|\psi\|_{C^{\scriptstyle k}}\|f\|_{\W[s,p]} \]
 where $k$ is the smallest integer greater than or equal to $s$.
\end{lemma}

 \section{Finite-energy paths are embedded}\label{sect:beta}
\renewcommand{\E}[1][p,q]{\ensuremath{\mathrm{TP}^{(#1)}}}

Let us indicate how to adapt the respective arguments presented in~\cite[Sect.~2]{SM7}
to get

\begin{theorem}[Embeddedness for $p\ge q+2$ {\cite[Thm.~1.1, Prop.~4.1]{SM7}}]\label{thm:topo}\hfill\\
 Let $\g\in C^{0,1}(\R/\Z,\R^n)$ be parametrized by arc-length with $\E(\g)<\infty$ for $p\ge q+2$. \\
 Then the image of $\g$ is a one-dimensional topological manifold, possibly with boundary, embedded in $\R^n$.
 In the case that $p > q+2$ this manifold is even of class $C^{1+\kappa}$ 
 for $\kappa = \frac {p-q-2}{q+4}$.
\end{theorem}

To this end it is sufficient to change just a few lines in the proof
of~\cite[Lem.~2.1]{SM7}. However, we add some more details
for the readers' convenience.

We briefly introduce some notation that will be used in the
statements below and refer to~\cite{SM7} for further details.
The \emph{beta numbers} introduced by Jones are defined via
\[ \beta_{\g}(x,r) := \inf\sets{\sup_{y\in\image\g\cap B_{r}(x)}\frac{\dist(y,G)}r}{G\text{ is a straight line through }x}. \]
For $\g(s)\ne\g(t)$ we denote the straight line through $\g(s)$ and $\g(t)$ by
\[ G(s,t) := \g(t)+\R\br{\g(s)-\g(t)}. \]
The $\delta$-neighborhood of some closed set $F$ is denoted by
\[ U_{\delta}(F) := \sets{x\in\R^{n}}{\dist(x,F)<\delta}. \]

\begin{lemma}[{}{\cite[Lem.~2.1]{SM7}}]\label{lem:SM7.2.1}
 For $p\ge q+2$ there is some $c_{p,q}>0$
 such that if $\g\in C^{0,1}\rzd$, $\eps\in(0,\frac1{200})$, and
 $d\in(0,\diam\image\g)$
 satisfy
 \begin{equation}\label{eq:SM7.2.9}
  \eps^{q+4}d^{2+q-p}\ge c_{p,q}\E(\g),
 \end{equation}
 then
 \[ \image\g\cap B_{2d}(\g(s))\subset U_{20\eps d}(G(s,t)) \]
 holds for any two points $\g(s)$, $\g(t)$ with $\abs{\g(s)-\g(t)}=d$.
 In particular,
 \[ \sup_{x\in\image\g}\beta_{\g}(x,2d)\le 10\eps. \]
\end{lemma}

Having this lemma, we follow exactly the line of argument in \cite{SM7}.
An immediate consequence of Lemma~\ref{lem:SM7.2.1} is then following
corollary, which guarantees a certain decay of Jones' beta numbers
if $p>q+2$:

\begin{corollary}[{}{\cite[Cor.~2.2]{SM7}}]\label{cor:SM7.2.2}
 For $p\ge q+2$ there are $\tilde c_{p,q},\delta_{p,q}>0$
 such that if
 \[ \E(\g)^{\tfrac1{q+4}}d^{\kappa}<\delta_{p,q} \qquad\text{where}\qquad
    \kappa = \tfrac{p-q-2}{q+4} \]
 for $\g\in C^{0,1}\rzd$ and $0<d\ll1$,
 then
 \[ \sup_{x\in\image\g}\beta_{\g}(x,2d)\le \tilde c_{p,q}\E(\g)^{\tfrac1{q+4}}d^{\kappa}. \]
\end{corollary}

\begin{proof}[Lemma~\ref{lem:SM7.2.1}]
 As the quantities in the claim do not depend on the actual para\-metrization of~$\g$, we may assume that $\g$ is parametrized by arc-length and set
 \begin{align*}
 A_{d}(s,\eps) &:= \sett{\tau\in\R/\Z}{\g(\tau)\in B_{\eps^{2}d}(\g(s))}, \\
 X_{d}(s,t,\eps) &:= \sett{\s\in A_{d}(s,\eps)}{\exists\; \dg(\s):
 \sphericalangle\br{\dg(\s),\g(t)-\g(s)}\in\bv{\tfrac\eps{10},\pi-\tfrac\eps{10}}}, \\
 N_{d}(s,t,\eps) &:= A_{d}(s,\eps)\setminus X_{d}(s,t,\eps).
 \end{align*}
 From~\cite[Eqn.~(2.10), (2.11)]{SM7} we infer
 for $\s\in X_{d}(s,t,\eps)$ and $\tau\in A_{d}(t,\eps)$
 \begin{subequations}
 \begin{align}
  \abs{\g(\s)-\g(\tau)} &\in \br{d(1-2\eps^{2}),d(1+2\eps^{2})}, \\
  \dist\br{\g(\tau),\ell(\s)} &\ge \frac{\eps d}{25}.
 \end{align}
 \end{subequations}
 From~\eqref{eq:ratio} we deduce
 \[ \frac1{\tilde r_{\g}^{(p,q)}(\s,\tau)} \ge c(p,q)\eps^{q}d^{q-p}. \]
 By $\abs{A_{d}(s,\eps)}\ge2\eps^{2}d$ we arrive at
 \[ \E(\g)\ge\iint_{X_{d}(s,t,\eps)\times A_{d}(s,\eps)}
    \frac{\d\s\d\tau}{\tilde r_{\g}^{(p,q)}(\s,\tau)}
    \ge c(p,q)\abs{X_{d}(s,t,\eps)}\eps^{q+2}d^{1+q-p}. \]
 As the assumption $\abs{X_{d}(s,t,\eps)}\ge\tfrac12\eps^{2}d$
 is not consistent with~\eqref{eq:SM7.2.9}
 for a sufficiently large choice of $c_{p,q}>0$, we obtain
 \begin{equation}\label{eq:N_d}
  \abs{N_{d}(s,t,\eps)} \ge \tfrac32\eps^{2} d.
 \end{equation}
 Supposing $\g(\tau)\in B_{2d}(\g(s))\setminus U_{20\eps d}(G(s,t))$,
 $\s\in N_{d}(s,t,\eps)$, and $\tau_{1}\in A_{d}(\tau,\eps)$
 yields by~\cite[Proof of Lemma~2.1, Step~2]{SM7}
 \[ \dist(\g(\tau_{1}),\ell(\s)) \ge 18\eps d \]
 which again gives
 \[ \frac1{\tilde r_{\g}^{(p,q)}(\s,\tau_{1})} \ge \tilde c(p,q)\eps^{q}d^{q-p} \]
 and
 \[ \E(\g)\ge\iint_{N_{d}(s,t,\eps)\times A_{d}(\tau,\eps)}
    \frac{\d\s\d\tau_{1}}{\tilde r_{\g}^{(p,q)}(\s,\tau_{1})}
    \ge \tilde c(p,q)\abs{N_{d}(s,t,\eps)}\eps^{q+2}d^{1+q-p}. \]
 Applying~\eqref{eq:N_d} and increasing~$c_{p,q}$ if necessary,
 this contradicts~\eqref{eq:SM7.2.9}.
\end{proof}

Revisiting the proof of Lemma~\ref{lem:SM7.2.1} we infer as in~\cite{SM7}
the following result for the critical case $p=q+2$:

\begin{lemma}[{}{\cite[Lem.~2.3]{SM7}}]\label{lem:SM7.2.3}
 There is some $c_{q}>0$ with
 \[ \sup_{x\in\image\g}\beta_{\g}(x,2d)\le c_{q}\omega_{q}(d) \]
 for $\g\in C^{0,1}\rzd$
 parametrized by arc-length, $\E[q+2,q](\g)<\infty$,
 and $0<d\ll1$,
 where $\omega_{q}(d)$ denotes the supremum of
 \[ \br{\iint_{A\times B}\frac{\d s\d t}{\tilde r_\g^{(q+2,q)}(s,t)}}^{\tfrac1{q+4}} \]
 taken over all pairs of subsets $A,B\subset\R/\Z$
 with $\abs A,\abs B\le\frac d{100}$.
\end{lemma}

\begin{proof}[\empty]\emph{Sketch of the proof of Theorem~\ref{thm:topo}.}
According to~\cite[Theorem~1.4]{SM7}, the image of any arc-length parametrized
curve $\g\in C^{0,1}\rzd$ with
\begin{equation}
 \sup_{x\in\image\g}\beta_{\g}(x,d)\le\omega(d),
\end{equation}
where $\omega:[0,1]\to[0,\infty)$
is some continuous non-decreasing function 
with $\omega(0)=0$,
is a one-dimensional submanifold
of~$\R^{n}$, possibly with boundary.

If now $p>q+2$ we get from Lemma~\ref{lem:SM7.2.1} that
\begin{equation*}
 \beta_{\g}(x,2d) \leq C d^\kappa
\end{equation*}
from which we deduce following exactly the 
arguments from \cite[Section~4]{SM7}, that the image of $\g$ is a submanifold
of class $C^\kappa$.
\end{proof}
\end{appendix}

\bibliography{projekt}
\bibliographystyle{abbrvhref}

\end{document}